\documentclass[12pt]{amsart}
\usepackage{amsmath}
\usepackage{amsthm}
\usepackage{amssymb}
\usepackage{amscd}
\usepackage{amsfonts}
\usepackage{amsbsy}
\usepackage{epsfig}
\usepackage{graphicx}
\usepackage{tikz}
\usepackage{tikz-cd}
\usetikzlibrary{arrows}
\usepackage{appendix}

\usepackage{xcolor}

\newtheorem{theorem}{Theorem}[section]
\newtheorem{corollary}[theorem]{Corollary}
\newtheorem{proposition}[theorem]{Proposition}
\newtheorem{lemma}[theorem]{Lemma}

\theoremstyle{definition}
\newtheorem{remark}[theorem]{Remark}

\newtheorem{definition}[theorem]{Definition}
\newtheorem{example}[theorem]{Example}

\def\ie{{\em i.e.,} }
\def\eg{{\em e.g.} }

\newfont\bbf{msbm10 at 12pt}
\def\eps{\varepsilon}
\def\phi{\varphi}
\def\R{{\mathbb R}}

\def\N{{\mathbb N}}

\def\theta{\vartheta}

\def\diam{\mbox{\rm diam}}

\def\Ent{\mbox{\rm Ent}}

\begin{document}

\title{Topological entropy of diagonal maps on inverse limit spaces}

\author{Ana\ Anu\v si\' c and Christopher Mouron}
\address[A.\ Anu\v{s}i\'c]{Departamento de Matem\'atica Aplicada, IME-USP, Rua de Mat\~ao 1010, Cidade Universit\'aria, 05508-090 S\~ao Paulo SP, Brazil}
\email{anaanusic@ime.usp.br}
\address[C.\ Mouron]{Rhodes College, 2000 North Parkway, Memphis, TN 38112, USA}
\email{mouronc@rhodes.edu}
\thanks{AA was supported by grant 2018/17585-5, S\~ao Paulo Research Foundation (FAPESP)}
\date{\today}


\begin{abstract}
	We give an upper bound for the topological entropy of maps on inverse limit spaces in terms of their set-valued components. In a special case of a diagonal map on the inverse limit space $\underleftarrow{\lim}(I,f)$, where every diagonal component is the same map $g\colon I\to I$ which strongly commutes with $f$ (\ie $f^{-1}\circ g=g\circ f^{-1}$), we show that the entropy equals $\max\{\Ent(f),\Ent(g)\}$. As a side product, we develop some techniques for computing topological entropy of set-valued maps.
\end{abstract}

\maketitle

\section{Introduction}\label{sec:intro}

Topological entropy is one of the most popular (topological) measures of complexity of a dynamical system. It was introduced by Adler, Konheim and  McAndrews \cite{AdKoMcA} for compact topological spaces, and further refined by Bowen \cite{Bowen-ent} and Dinaburg \cite{Dinaburg} in case of metrizable spaces. We are interested in describing entropy of maps on complicated spaces in terms of much simpler, one-dimensional maps. 

The spaces we study can be described as {\em inverse limits} on compact, connected, metric spaces (often called {\em continua}) $X_i$, $i\geq 0$, with continuous and onto {\em bonding functions} $f_i\colon X_i\to X_{i-1}$, $i\in\N$. That is, we define
$$\underleftarrow{\lim}(X_i,f_i):=\{(x_0,x_1,x_2,\ldots): f_i(x_i)=x_{i-1}, i\in\N\}\subseteq \prod_{i=0}^{\infty} X_i.$$
The space $\underleftarrow{\lim}(X_i,f_i)$, equipped with the product topology (\ie the smallest topology in which all projections $\pi_i\colon \underleftarrow{\lim}(X_i,f_i)\to X_i$ are continuous), is also a continuum. Normally, one takes spaces $X_i$ and maps $f_i$ to be simple, and uses the inverse limit structure to generate a very complicated space. For example, every one-dimensional continuum can be described as an inverse limit on graphs, with bonding maps which are piecewise linear, see \cite{Freudenthal}. The factor spaces $X_i$ correspond to {\em nerves} of finite covers of $\underleftarrow{\lim}(X_i,f_i)$, and bonding maps $f_i$ describe behaviors of nerves of refinements.

We are interested in calculating the entropy of maps on an inverse limit space $\underleftarrow{\lim}(X_i,f_i)$. For example, some maps on $\underleftarrow{\lim}(X_i,f_i)$ can be described by their ``straight down" components. That is, assume that there is a sequence of maps $g_i\colon X_i\to X_i$ such that $g_i\circ f_{i+1}=f_i\circ g_i$ for every $i\geq 0$, see the commutative diagram in Figure~\ref{fig:straightdown}. 

\begin{figure}[!ht]
	\begin{tikzpicture}[->,>=stealth',auto, scale=2]
	\node (1) at (0,1) {$X_0$};
	\node (2) at (0,0) {$X_0$};
	\node (3) at (1,1) {$X_1$};
	\node (4) at (1,0) {$X_1$};
	\node (5) at (2,1) {$X_2$};
	\node (6) at (2,0) {$X_2$};
	\node (7) at (3,1) {$X_3$};
	\node (8) at (3,0) {$X_3$};
	\node (9) at (4,1) {$\ldots$};
	\node (10) at (4,0) {$\ldots$};
	\draw [->] (3) to (1);
	\draw [->] (5) to (3);
	\draw [->] (7) to (5);
	\draw [->] (9) to (7);
	\draw [->] (4) to (2);
	\draw [->] (6) to (4);
	\draw [->] (8) to (6);
	\draw [->] (10) to (8);
	\draw [->] (1) to (2);
	\draw [->] (3) to (4);
	\draw [->] (5) to (6);
	\draw [->] (7) to (8);
	
	\node at (0.55,1.15) {\tiny $f_1$};
	\node at (1.55,1.15) {\tiny $f_2$};
	\node at (2.55,1.15) {\tiny $f_3$};
	\node at (3.55,1.15) {\tiny $f_4$};
	\node at (0.55,-0.2) {\tiny $f_1$};
	\node at (1.55,-0.2) {\tiny $f_2$};
	\node at (2.55,-0.2) {\tiny $f_3$};
	\node at (3.55,-0.2) {\tiny $f_4$};
	
	\node at (-0.1,0.5) {\tiny $g_0$};
	\node at (0.9,0.5) {\tiny $g_1$};
	\node at (1.9,0.5) {\tiny $g_2$};
	\node at (2.9,0.5) {\tiny $g_3$};
	\end{tikzpicture}
	\caption{Commutative diagram from Ye's theorem.}
	\label{fig:straightdown}
\end{figure} 

Then we can define a map $\Psi\colon \underleftarrow{\lim}(X_i,f_i)\to \underleftarrow{\lim}(X_i,f_i)$ as $\Psi((x_i)_{i=0}^{\infty})=((g_i(x_i))_{i=0}^{\infty})$. It was shown by Ye in \cite{Ye} that  $\lim_{i\to\infty}\Ent(g_i)$ exists and $\Ent(\Psi)=\lim_{i\to\infty}\Ent(g_i)$. A special case of this result goes back to Bowen \cite{Bowen}. He showed that if there are $X,f$ such that $X_i=X$, and $f_i=f$ for all $i$, then the entropy of the {\em shift homeomorphism} (also called natural extension) $\hat f\colon\underleftarrow{\lim}(X,f)\to\underleftarrow{\lim}(X,f)$, given by $\hat f((x_i)_{i=0}^{\infty})=(f(x_i))_{i=0}^{\infty}$, equals the entropy of $f$. Ye's result was used by  Mouron in \cite{Mou-psarc}, to show that the entropy of shift homeomorphisms on the {\em pseudo-arc} is either $0$ or $\infty$.

However, not all maps on $\underleftarrow{\lim}(X_i,f_i)$ can be represented by their ``straight-down" components, see  Miodusewski's characterization in \cite{Miodus}. For example, the map can be given by its ``diagonal components" as in  Figure~\ref{fig:diagonal}, or diagrams in Figure~\ref{fig:straightdown} and Figure~\ref{fig:diagonal} can only $\eps_i$-commute, for $\eps_i\to 0$ as $i\to\infty$. We show that there actually is a way to recover ``straight-down" components of any map on an inverse limit, with only a small price to pay - the ``straight-down" components are allowed to be set-valued functions. Given a continuum $X$, let $2^X$ denote the set of all non-empty closed subsets of $X$. We prove the following:

{\bf Theorem\,3.4.} Let $\Psi\colon\underleftarrow{\lim}(X_i, f_i)\to\underleftarrow{\lim}(X_i, f_i)$ be a continuous function. For $i\geq 0$ define $\psi_i\colon X_i\to 2^{X_i}$ as $\psi_i(x)=\pi_i\circ\Psi\circ\pi_i^{-1}(x)$. Then $\Ent(\Psi)\leq \liminf_{i\rightarrow\infty}\Ent(\psi_i)$.

The definition of entropy of a set-valued map is an extension of the standard Bowen's definition for single-valued maps, introduced and explored in a number of recent papers, see \eg \cite{ErcegKen,KelTen,COAR,AlKel}, either from a purely dynamical perspective, or to understand so-called {\em generalized inverse limits} introduced by Mahavier \cite{Mah}, and developed by Ingram and Mahavier \cite{IngMah}. In generalized inverse limits one allows the bonding maps to be set-valued, so it makes sense to have  a set-valued natural extension as in \cite{JudyVan}. Generalized inverse limits are beyond the scope of this paper. We note, however, that the results of this paper provide an interesting interplay between the standard notion of a map and inverse limit, set-valued maps and their dynamics, and generalized inverse limits.

Given a map $\Psi$ on an inverse limit $\underleftarrow{\lim}(X_i,f_i)$, we can find an upper bound for its entropy by computing the entropy of ``straight-down" components $\psi_i\colon X_i\to 2^{X_i}$. One might hope that $\lim_{i\to\infty}\Ent(\psi_i)$ exists and equals the entropy of $\Psi$ as in Ye's result, but that is unfortunately not true in general, as we show in the appendix. However, there is a wide class for which it is indeed true, which we discuss next.

Let $\Psi\colon\underleftarrow{\lim}(X_i,f_i)\to \underleftarrow{\lim}(X_i,f_i)$ be a {\em diagonal} map given by a commutative diagram as in Figure~\ref{fig:diagonal}. That is, there are maps $g_i\colon X_i\to X_{i-1}$, such that $g_i\circ f_{i+1}=f_i\circ g_{i+1}$ for every $i\in\N$, and $\Psi((x_0,x_1,x_2,\ldots))=(g_1(x_1),g_2(x_2),\ldots)$. 

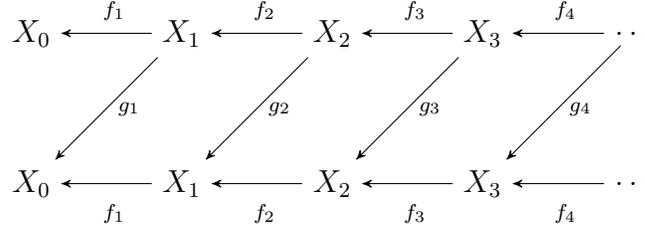
\begin{figure}[!ht]
	\begin{tikzpicture}[->,>=stealth',auto, scale=2]
	\node (1) at (0,1) {$X_0$};
	\node (2) at (0,0) {$X_0$};
	\node (3) at (1,1) {$X_1$};
	\node (4) at (1,0) {$X_1$};
	\node (5) at (2,1) {$X_2$};
	\node (6) at (2,0) {$X_2$};
	\node (7) at (3,1) {$X_3$};
	\node (8) at (3,0) {$X_3$};
	\node (9) at (4,1) {$\ldots$};
	\node (10) at (4,0) {$\ldots$};
	\draw [->] (3) to (1);
	\draw [->] (5) to (3);
	\draw [->] (7) to (5);
	\draw [->] (9) to (7);
	\draw [->] (4) to (2);
	\draw [->] (6) to (4);
	\draw [->] (8) to (6);
	\draw [->] (10) to (8);
	\draw [->] (3) to (2);
	\draw [->] (5) to (4);
	\draw [->] (7) to (6);
	\draw [->] (9) to (8);
	
	\node at (0.55,1.15) {\tiny $f_1$};
	\node at (1.55,1.15) {\tiny $f_2$};
	\node at (2.55,1.15) {\tiny $f_3$};
	\node at (3.55,1.15) {\tiny $f_4$};
	\node at (0.55,-0.2) {\tiny $f_1$};
	\node at (1.55,-0.2) {\tiny $f_2$};
	\node at (2.55,-0.2) {\tiny $f_3$};
	\node at (3.55,-0.2) {\tiny $f_4$};
	
	\node at (0.65,0.5) {\tiny $g_1$};
	\node at (1.65,0.5) {\tiny $g_2$};
	\node at (2.65,0.5) {\tiny $g_3$};
	\node at (3.65,0.5) {\tiny $g_4$};
	\end{tikzpicture}
	\caption{Commutative diagram in the construction of a diagonal map.}
	\label{fig:diagonal}
\end{figure} 

Then Theorem~\ref{thm:liminf} gives $\Ent(\Psi)\leq\liminf_{i\to\infty}\Ent(g_i\circ f_i^{-1})$. If we additionally assume that $g_{i+1}\circ f_{i+1}^{-1}=f_i^{-1}\circ g_i$, for every $i\in\N$, then we show in Proposition~\ref{prop:entdiag} that $\lim_{i\to\infty}\Ent(g_i\circ f_i^{-1})$ exists and $\Ent(\Psi)=\lim_{i\rightarrow \infty}\Ent(g_i\circ f_i^{-1})$. Thus in this case the entropy of the map is completely determined by its (set-valued) straight-down components.

We note that diagonal maps on inverse limits have already proved to be very useful. For example, they were used to construct an example of a tree-like continuum without a fixed point property (\ie for which there is a self-map without a fixed point) by Oversteegen and Rogers \cite{OvRog, OvRog2}, and Hoehn and Hern\'andez-Guti\'errez \cite{HoHerGut}. Also, diagonal map was used by Mouron in \cite{Mou2} to give an example of an exact map on the pseudo-arc.

In the remainder of the paper we study in more detail the property $g_{i+1}\circ f_{i+1}^{-1}=f_i^{-1}\circ g_i$ from Proposition~\ref{prop:entdiag}, and how to compute $\Ent(g_i\circ f_i^{-1})$ in that case. In particular, we restrict to a case when there are maps $f,g\colon I\to I$ such that $f_i=f$, and $g_i=g$ for every $i\in\N$, where $I=[0,1]$ is the unit interval. Then the diagonal map is $\Psi\colon\underleftarrow{\lim}(I,f)\to\underleftarrow{\lim}(I,f)$, $\Psi((x_0,x_1,x_2,\ldots))=(g(x_1),g(x_2),\ldots)$. In particular, $\Ent(\Psi)\leq\Ent(g\circ f^{-1})$, and if $g\circ f^{-1}=f^{-1}\circ g$, then $\Ent(\Psi)=\Ent(g\circ f^{-1})$.

Note that if $f,g\colon X\to X$ commute (\ie $f\circ g=g\circ f$), then $g\circ f^{-1}\subseteq f^{-1}\circ g$. Commuting maps $f,g$ for which $g\circ f^{-1}=f^{-1}\circ g$ are called {\em strongly commuting}, and were introduced in \cite{AM-stronglycomm}. There, we show that if $f,g\colon I\to I$ are strongly commuting and piecewise monotone, then they can be decomposed into finitely many common invariant subintervals on which at least one of $f,g$ is an open map. Actually, on each common invariant subinterval either both $f,g$ are open, or at least one is monotone. So it turns out that, in order to compute $\Ent(g\circ f^{-1})$ for strongly commuting interval maps $f,g$ as above, we only need to compute $\Ent(F_n\circ F_m^{-1})$, where $n,m\geq 2$, and $F_n$ denotes an open interval map with $n-1$ critical points. We show the following:

{\bf Theorem\,4.20.} Let $F_n\circ F_m^{-1}=F_m^{-1}\circ F_n$, $n,m\geq 2$. Then $$\Ent(F_n\circ F^{-1}_m)=\max\{\log(n), \log(m)\}.$$

For example, this shows that for every odd $n\geq 2$, there is an entropy $\log(n)$ map on the {\em Knaster continuum} $\underleftarrow{\lim}(I,T_2)$, where $T_2\colon I\to I$ is the full tent map given by $T_2(x)=\min\{2x,2(1-x)\}$, for $x\in I$. It is known that all the homeomorphisms on the Knaster continuum have entropy $k\log(2)$, for $k\in\N$, see \cite{BrSt}.

In particular, Theorem~\ref{thm:main} in combination with \cite{AM-stronglycomm} implies that if $f,g\colon I\to I$ are strongly commuting and piecewise monotone, then $\Ent(g\circ f^{-1})=\max\{\Ent(f),\Ent(g)\}$, see Corollary~\ref{cor:main}. The methods we use resemble the ones used to compute entropy of  piecewise monotone interval maps by Misiurewicz and Szlenk \cite{MisSl}. In particular, we note that for set-valued maps of the form $g\circ f^{-1}$ for strongly commuting interval maps $f,g$, the entropy can be computed by counting the number of ``monotone branches", or cardinality of image and preimage sets. This methods have a potential to be generalized further.

The outline of the paper is as follows: we give preliminaries on set-valued functions, entropy, and inverse limits in Section~\ref{sec:prel}. Then we introduce ``straight-down" components of a map on an inverse limit and prove Theorem~\ref{thm:liminf} and Proposition~\ref{prop:entdiag} in Section~\ref{sec:entropy}. Section~\ref{sec:main} is technically most difficult part of the paper in which we compute $\Ent(F_m\circ F_n^{-1}=F_n^{-1}\circ F_m)$, where $n,m\geq 2$, and $F_n\colon I\to I$ denotes an open map with $n-1$ critical points, \ie points at which $F_n$ is locally not monotone. The main theorem in that section is Theorem~\ref{thm:main}.  We note that the techniques introduced in this section work in much higher generality. Finally, in Section~\ref{sec:stronglycomm} we recall some facts about strongly commuting interval maps $f,g$ from \cite{AM-stronglycomm}, and show in Corollary~\ref{cor:main} that if $f,g\colon I\to I$ are piecewise monotone, and strongly commuting maps, then $\Ent(g\circ f^{-1})=\max\{\Ent(f),\Ent(g)\}$. In the appendix we give an example of an inverse limit and a map which show that the condition $g_{i+1}\circ f_{i+1}^{-1}=f_i^{-1}\circ g_i$, $i\in\N$ from Proposition~\ref{prop:entdiag} is indeed necessary.

\section{Preliminaries}\label{sec:prel}
Given two topological spaces $X, Y$, a continuous function $f\colon X\to Y$ will be referred to as a {\em map}. We will restrict this paper to compact, connected, metrizable spaces (also called {\em continua}), and often, only the unit interval $I=[0,1]$. Map $f\colon I\to I$ is called {\em piecewise monotone} if there is $n\in\N$, and points $c_0=0<c_1<\ldots<c_{n-1}<c_n=1$ such that $f|_{[c_i,c_{i+1}]}$ is one-to-one for every $i\in\{0,\ldots,n-1\}$. In that case, points $c_i$, for $i\in\{1,\ldots,n-1\}$ are called {\em critical points} of $f$. Furthermore, we say that $f\colon I\to I$ is {\em open} if $f(U)$ is open in $I$ for every open set $U\subset I$. Note that is $f$ is piecewise monotone and open, then $f$ maps $[c_i,c_{i+1}]$ onto $I$ for every $i\in\{0,\ldots,n-1\}$.

{\em Set-valued function} from $X$ to $Y$ is a function $F\colon X\to 2^Y$, where $2^Y$ denotes the set of all non-empty closed subsets of $Y$. We always assume that set-valued functions are {\em upper semi-continuous}, \ie the {\em graph} of $F$,
$$\Gamma(F)=\{(x,y): y\in F(x)\}$$
is closed in $X\times Y$. When there is no confusion, we will often abuse the notation and denote set-valued functions as $F\colon X\to Y$.

Let $F\colon X\to X$ be a set-valued function. For $n\in\N$, an {\em $n$-orbit of F} is every $n$-tuple $(x_1, \ldots, x_n)$ such that $x_{i+1}\in F(x_i)$ for every $1\leq i<n$. Denote by $Orb_n(F)$ the set of all $n$-orbits of $F$.

For $n\in\N$ and $\eps>0$, we say that a set $S\subset Orb_n(F)$ is {\em $(n, \eps)$-separated} if for every $(x_1,\ldots,x_n), (y_1,\ldots, y_n)\in S$ there exists $i\in\{1, \ldots, n\}$ such that $d_X(x_i,y_i)>\eps$, where $d_X$ denotes the metric on $X$. Let $s_{n,\eps}(F)$ denote the largest cardinality of an $(n, \eps)$-separated set.

The {\em entropy} of $F$ is defined as $$\Ent(F)=\lim_{\eps\to 0}\limsup_{n\to\infty}\frac{1}{n}\log(s_{n,\eps}(F)).$$

Similarly, for $n\in\N$ and $\eps>0$, we say that a set $S\subset Orb_n(F)$ is {\em $(n, \eps)$-spanning} if for every $(x_1, \ldots, x_n)\in Orb_n(F)$ there exists $(y_1, \ldots, y_n)\in S$ such that $d_X(x_i, y_i)<\eps$ for every $i\in\{1, \ldots, n\}$. The smallest cardinality of an $(n, \eps)$-spanning set is denoted by $r_{n, \eps}(F)$. It can be shown (see \eg \cite{AlKel}) that

$$\Ent(F)=\lim_{\eps\to 0}\limsup_{n\to\infty}\frac{1}{n}\log(r_{n,\eps}(F)).$$

Let $X_i$, $i\geq 0$ be continua and let $f_i\colon X_i\to X_{i-1}$, be continuous functions, $i\geq 1$. The {\em inverse limit space} of the system $(X_i, f_i)$ is defined as
$$\underleftarrow{\lim}(X_i, f_i):=\{(x_0, x_1, x_2, \ldots): f_i(x_i)=x_{i-1}, i\geq 1\}\subset \prod_{i=0}^{\infty}X_i,$$
equipped with the product topology, \ie the smallest topology such that the {\em coordinate projections} $\pi_i\colon\underleftarrow{\lim}(X_i, f_i)\to X_i$, $i\geq 0$, are continuous. The space $\underleftarrow{\lim}(X_i, f_i)$ is again a continuum. For $i<j$ we denote by $f_i^j\colon X_j\to X_i$, $f_i^j=f_{i+1}\circ\ldots\circ f_j$. Note that $f_i^j\circ\pi_j=\pi_i$ for every $i<j$.

If we assume that $\diam(X_i)=1$ for every $i\geq 0$, then we can write the metric $d$ on $\underleftarrow{\lim}(X_i,f_i)$ as
$$d((x_i)_{i=0}^{\infty},(y_i)_{i=0}^{\infty})=\sum_{i=0}^{\infty}\frac{d_{i}(x_i,y_i)}{2^i},$$
for every $(x_i)_{i=0}^{\infty},(y_i)_{i=0}^{\infty}\in \underleftarrow{\lim}(X_i,f_i)$, where $d_i\colon X_i\times X_i\to \R$ denotes a metric on $X_i$.

Given $\underleftarrow{\lim}(X_i, f_i)$, we are interested in the dynamical properties of maps $f\colon \underleftarrow{\lim}(X_i, f_i)\to \underleftarrow{\lim}(X_i, f_i)$, and in particular, their topological entropy. For example, if $X_i=X$ and $f_i=f$, then we can define the {\em shift homeomorphism} (also called natural extension) $\hat f\colon \underleftarrow{\lim}(X, f)\to\underleftarrow{\lim}(X, f)$ as
$$\hat f(x_0, x_1, x_2, \ldots)=(f(x_0), f(x_1), f(x_2), \ldots)=(f(x_0), x_0, x_1, \ldots).$$
The map can be visualized by the commutative diagram in Figure~\ref{fig:shift}.

\begin{figure}[!ht]
	\begin{tikzpicture}[->,>=stealth',auto, scale=2]
	\node (1) at (0,1) {$X$};
	\node (2) at (0,0) {$X$};
	\node (3) at (1,1) {$X$};
	\node (4) at (1,0) {$X$};
	\node (5) at (2,1) {$X$};
	\node (6) at (2,0) {$X$};
	\node (7) at (3,1) {$X$};
	\node (8) at (3,0) {$X$};
	\node (9) at (4,1) {$\ldots$};
	\node (10) at (4,0) {$\ldots$};
	\draw [->] (3) to (1);
	\draw [->] (5) to (3);
	\draw [->] (7) to (5);
	\draw [->] (9) to (7);
	\draw [->] (4) to (2);
	\draw [->] (6) to (4);
	\draw [->] (8) to (6);
	\draw [->] (10) to (8);
	\draw [->] (1) to (2);
	\draw [->] (3) to (4);
	\draw [->] (5) to (6);
	\draw [->] (7) to (8);
	
	\node at (0.55,1.15) {\tiny $f$};
	\node at (1.55,1.15) {\tiny $f$};
	\node at (2.55,1.15) {\tiny $f$};
	\node at (3.55,1.15) {\tiny $f$};
	\node at (0.55,-0.2) {\tiny $f$};
	\node at (1.55,-0.2) {\tiny $f$};
	\node at (2.55,-0.2) {\tiny $f$};
	\node at (3.55,-0.2) {\tiny $f$};
	
	\node at (-0.1,0.5) {\tiny $f$};
	\node at (0.9,0.5) {\tiny $f$};
	\node at (1.9,0.5) {\tiny $f$};
	\node at (2.9,0.5) {\tiny $f$};
	\end{tikzpicture}
	\caption{Commutative diagram realizing the action of the natural extension $\hat f$.}
	\label{fig:shift}
\end{figure}

It is not difficult to see (\eg \cite{Bowen}) that $\Ent(\hat f)=\Ent(f)$, so the entropy can be computed using the ``straight-down" components of $\hat f$. Ye generalized this fact to continuous maps on inverse limits which can be represented by a ``straight-down" commutative diagram:

\begin{theorem}[Theorem 3.1 in \cite{Ye}]
	Let $X_i$, $i\geq 0$ be continua and $f_i\colon X_i\to X_{i-1}$, $i\geq 1$, be continuous functions. Assume that there exist maps $g_i\colon X_i\to X_i$, such that $g_i\circ f_{i+1}=f_{i+1}\circ g_{i+1}$, for every $i\geq 0$ (see the commutative diagram in Figure~\ref{fig:straightdown}). Then for the map 
	$\Psi\colon\underleftarrow{\lim}(X_i, f_i)\to\underleftarrow{\lim}(X_i, f_i)$ given as 
	$$\Psi((x_0,x_1,x_2, \ldots))=(g_0(x_0), g_1(x_1), g_2(x_2), \ldots),$$ 
	it holds that $\Ent(\Psi)=\sup_i \Ent(g_i)$.
\end{theorem} 

\section{Maps on inverse limit spaces and their straight-down components}\label{sec:entropy}

In general there are many self-maps of $\underleftarrow{\lim}(X_i, f_i)$ which cannot be represented by commutative ``straight-down" diagrams as in Figure~\ref{fig:straightdown}. Here we show that every self-map of $\underleftarrow{\lim}(X_i, f_i)$ can be represented as a ``straight-down" diagram, but we need to allow the ``straight-down" functions to be set-valued. We then give an upper bound for the entropy using those set-valued maps. We also obtain a lower bond in particular cases.

Let $\Psi\colon\underleftarrow{\lim}(X_i, f_i)\to\underleftarrow{\lim}(X_i, f_i)$ be a continuous function. For $i\geq 0$, define (possibly set-valued functions) $\psi_i\colon X_i\to X_i$ as $\psi_i(x)=\pi_i\circ\Psi\circ\pi_i^{-1}(x)$, for every  $x\in X_i$. Then we obtain the diagram in Figure~\ref{fig:generalYe}.
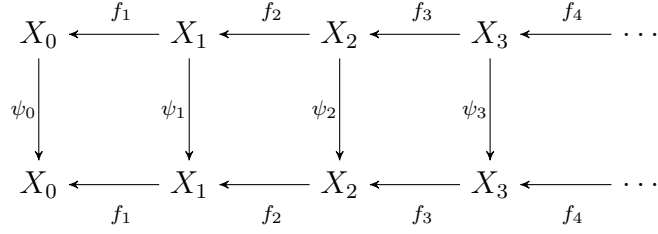
\begin{figure}[!ht]
	\begin{tikzpicture}[->,>=stealth',auto, scale=2]
	\node (1) at (0,1) {$X_0$};
	\node (2) at (0,0) {$X_0$};
	\node (3) at (1,1) {$X_1$};
	\node (4) at (1,0) {$X_1$};
	\node (5) at (2,1) {$X_2$};
	\node (6) at (2,0) {$X_2$};
	\node (7) at (3,1) {$X_3$};
	\node (8) at (3,0) {$X_3$};
	\node (9) at (4,1) {$\ldots$};
	\node (10) at (4,0) {$\ldots$};
	\draw [->] (3) to (1);
	\draw [->] (5) to (3);
	\draw [->] (7) to (5);
	\draw [->] (9) to (7);
	\draw [->] (4) to (2);
	\draw [->] (6) to (4);
	\draw [->] (8) to (6);
	\draw [->] (10) to (8);
	\draw [->] (1) to (2);
	\draw [->] (3) to (4);
	\draw [->] (5) to (6);
	\draw [->] (7) to (8);
	
	\node at (0.55,1.15) {\tiny $f_1$};
	\node at (1.55,1.15) {\tiny $f_2$};
	\node at (2.55,1.15) {\tiny $f_3$};
	\node at (3.55,1.15) {\tiny $f_4$};
	\node at (0.55,-0.2) {\tiny $f_1$};
	\node at (1.55,-0.2) {\tiny $f_2$};
	\node at (2.55,-0.2) {\tiny $f_3$};
	\node at (3.55,-0.2) {\tiny $f_4$};
	
	\node at (-0.1,0.5) {\tiny $\psi_0$};
	\node at (0.9,0.5) {\tiny $\psi_1$};
	\node at (1.9,0.5) {\tiny $\psi_2$};
	\node at (2.9,0.5) {\tiny $\psi_3$};
	\end{tikzpicture}
	\caption{Decomposing the map $\protect\Psi$ into ``straight-down" (possibly set-valued) components $\protect\psi_i$.}
	\label{fig:generalYe}
\end{figure}

Note that if $y\in\underleftarrow{\lim}(X_i, f_i)$ is such that $\pi_j(y)=x\in X_j$, then for $i<j$ it holds that $\pi_i(y)=f_i^j\circ\pi_j(y)=f_i^j(x)$. It follows that $\pi_j^{-1}(x)\subset\pi_i^{-1}(f_i^j(x))$ for every $x\in X_j$. Thus also $f_i^j\circ\psi_j(x)=f_i^j\circ\pi_j\circ\Psi\circ\pi_j^{-1}(x)=\pi_i\circ\Psi\circ\pi_j^{-1}(x)\subset\pi_i\circ\Psi\circ\pi_i^{-1}(f_i^j(x))=\psi_i\circ f_i^j(x)$. 
So diagram in Figure~\ref{fig:generalYe} commutes in a sense that $f_i^j\circ\psi_j(x)\subset\psi_i\circ f_i^j(x)$ for every $i<j$ and $x\in X_j$.

We want to compute the entropy of $\Psi$ in terms of (set-valued) maps $\psi_i$, $i\geq 0$. We first prove the following lemma. Recall that we assume $\mbox{diam}(X_i)=1$, $d_i$ denotes the metric on $X_i$, and the metric on $\underleftarrow{\lim}(X_i, f_i)$ is given by $d(x,y)=\sum_{i=0}^{\infty}\frac{d_i(x_i,y_i)}{2^i}$, for every $x,y\in\underleftarrow{\lim}(X_i,f_i)$. 

\begin{lemma}\label{lem:yadda}
	Let $X= \underleftarrow{\lim}(X_i, f_i) $ with each $X_i$ compact. Given $\eps>0$, there exists an $m=m(\eps)$ such that for every $i>m$ there is an $\eps_i>0$ such that if $x,y\in X$ with $d(x,y)\geq \eps$, then $d_i(x_i,y_i)\geq \eps_i$.
\end{lemma}
\begin{proof} Given $\eps>0$, there exists $m$ such that $\sum_{i=m+1}^{\infty}\frac{1}{2^i}<\eps/2$. It follows that if $d(x,y)\geq \eps$, then $\sum_{i=0}^{m}\frac{d_i(x_i,y_i)}{2^i}\geq \eps/2$. Thus, it follows that there exists $k\in\{0,...,m\}$ such that $d_k(x_k,y_k)\geq \eps/4$. Note that, given any  $k\in\{0,...,m\}$ and $i>m$, it follows from uniform continuity of $f_k^i$ that there exists $\delta_i^k>0$ such that if $d_k(x_k,y_k)=d_k(f_k^i(x_i),f_k^i(y_i))\geq \eps/4$ then $d_i(x_i,y_i)\geq \delta_i^k$. For $\eps_i:=\min\{\delta_i^k: k\in \{0,...,m\}\}$ it holds that $\delta_i^k\geq \eps_i$ for every $k\in\{0,\ldots,m\}$, and we have the prescribed value.
\end{proof}

\begin{proposition}
	Let $\Psi\colon\underleftarrow{\lim}(X_i, f_i)\to\underleftarrow{\lim}(X_i, f_i)$ be a continuous function. For $i\geq 0$ define (possibly set-valued functions) $\psi_i\colon X_i\to X_i$ as $\psi_i(x)=\pi_i\circ\Psi\circ\pi_i^{-1}(x)$. If $(\xi_1,\xi_2,...,\xi_n)$ is an orbit of $\Psi$, then  $(\pi_i(\xi_1),\pi_i(\xi_2),...,\pi_i(\xi_n))$ is an orbit of $\psi_i$.
\end{proposition}
\begin{proof} Let $k\in\{1,\ldots, n-1\}$, and notice that $\xi_k\in \pi_i^{-1}\circ \pi_i(\xi_k)$. Hence, $\xi_{k+1}=\Psi(\xi_k)\in \Psi\circ \pi_i^{-1}\circ \pi_i(\xi_k)$. Thus, $\pi_i(\xi_{k+1})=\pi_i\circ\Psi(\xi_k)\in \pi_i\circ\Psi\circ \pi_i^{-1}\circ \pi_i(\xi_k)=\psi_i(\pi_i(\xi_k))$.  We conclude that $(\pi_i(x_1),\pi_i(x_2),...,\pi_i(x_k),\pi_i(x_{k+1}))$ is an orbit of $\psi_i$.
\end{proof}

For $S\subset \mbox{Orb}_n(\Psi)$, denote by $\pi_i(S)=\{(\pi_i(\xi_1),\pi_i(\xi_2),...,\pi_i(\xi_n)): (\xi_1, \xi_2,..., \xi_n)\in  S\}$.

\begin{proposition}\label{prop:yadda2}
	Let $\Psi\colon\underleftarrow{\lim}(X_i, f_i)\to\underleftarrow{\lim}(X_i, f_i)$ be a continuous function. For $i\geq 0$ define (possibly set-valued functions) $\psi_i\colon X_i\to X_i$ as $\psi_i(x)=\pi_i\circ\Psi\circ\pi_i^{-1}(x)$. If $S$ is an $(n,\eps)$-separated set for $\Psi$, then there exists $m$ such that for for all $i>m$, there exists an $\eps_i>0$ such that $\pi_i(S)$ is $(n,\eps_i)$-separated. Furthermore, $|S|=|\pi_i(S)|$.
\end{proposition}
\begin{proof}  
	Let $S$ be $(n,\eps)$ separated for $\Psi$. Then find $m=m(\eps)$ by Lemma~\ref{lem:yadda}. Let $i>m$ and find $\eps_i>0$ also by Lemma~\ref{lem:yadda}. For distinct $(\xi_1, \xi_2,...,\xi_n),(\mu_1, \mu_2,...,\mu_n)\in S$ there exists $k\in\{1,...,n\}$ such that $d(\xi_k,\mu_k)\geq \eps$. Then it follows that  $d_i(\pi_i(\xi_k),\pi_i(\mu_k))\geq \eps_i$. So, $\pi_i(S)$ is $(n,\eps_i)$-separated. It also follows that distinct elements of $S$ are taken to distinct elements of $\pi_i(S)$, so $|S|=|\pi_i(S)|$.
\end{proof}

\begin{theorem}\label{thm:liminf}
	Let $\Psi\colon\underleftarrow{\lim}(X_i, f_i)\to\underleftarrow{\lim}(X_i, f_i)$ be a continuous function. For $i\geq 0$ define (possibly set-valued functions) $\psi_i\colon X_i\to X_i$ as $\psi_i(x)=\pi_i\circ\Psi\circ\pi_i^{-1}(x)$. Then $\Ent(\Psi)\leq \liminf_{i\rightarrow\infty}\Ent(\psi_i)$.
\end{theorem}
\begin{proof} 
	
	Suppose on the contrary that there exists $c>0$ and a strictly increasing sequence of natural numbers,  $\{i_p\}_{p=1}^{\infty}$ such that $\mbox{Ent}(\psi_{i_p})\leq c< \mbox{Ent}(\Psi)$ for all $p$. Then there is an $\eps>0$ such that $\limsup_{n\rightarrow\infty}(1/n)\log(s_{n,\eps}(\Psi))>c$. For $n\in\N$, let $S_n$ be an $(n,\eps)$-separated set in $Orb_n(\Psi)$ such that $|S_n|=s_{n,\eps}(\Psi)$. It follows from Proposition~\ref{prop:yadda2} that there exist $\widehat{p}$ and $\widehat{\eps}=\eps_{i_{\widehat{p}}}>0$ such that $\widehat{S}_n=\pi_{i_{\widehat{p}}}(S_n)$ is an $(n,\widehat{\eps})$-separated set of $\psi_{i_{\widehat{p}}}$ and $|\widehat{S}_n|=|S_n|=s_{n,\eps}(\Psi)$. Now, it follows that \[\limsup_{n\to\infty}(1/n)\log(s_{n,\widehat{\eps}}(\psi_{i_{\widehat p}}))\geq \limsup_{n\rightarrow\infty}(1/n)\log(|\widehat{S}_n|)= \limsup_{n\rightarrow\infty}(1/n)\log(s_{n,\eps}(\Psi))>c.\] 
	Furthermore, note that for every $\eps'<\widehat{\eps}$ it holds that $s_{n,\eps'}(\psi_{i_{\widehat p}})\geq s_{n,\widehat{\eps}}(\psi_{i_{\widehat p}})$, so 
	$$\limsup_{n\rightarrow\infty}(1/n)\log(s_{n,\eps'}(\psi_{i_{\widehat p}}))\geq \limsup_{n\rightarrow\infty}(1/n)\log(s_{n,\widehat\eps}(\psi_{i_{\widehat p}}))>c.$$
	It follows that $\Ent(\psi_{i_{\widehat{p}}})>c$, and that is a contradiction.
\end{proof}

Specifically, we want to study the dynamical properties of the maps arising from commutative``diagonal" diagrams as in Figure~\ref{fig:diagonal}. We say that $\Psi\colon\underleftarrow{\lim}(X_i,f_i)\to\underleftarrow{\lim}(X_i, f_i)$ is a {\em diagonal map} if there exists continuous maps $g_i\colon X_i\to X_{i-1}$, $i\geq 1$, such that $g_i\circ f_{i+1}=f_i\circ g_{i+1}$ for every $i\geq 1$ (see the commutative diagram in Figure~\ref{fig:diagonal}) and such that $\Psi$ is given by 
$$\Psi((x_0, x_1, x_2, \ldots))=(g_1(x_1), g_2(x_2), g_3(x_3), \ldots).$$

We can recover the (set-valued) straight-down components by $\psi_i(x)=\pi_i\circ \Psi\circ\pi_i^{-1}(x)=g_{i+1}\circ f_{i+1}^{-1}(x)$ for $i\geq 0$, see Figure~\ref{fig:diagonalStraightDown}. With some extra assumptions (see the following proposition), we show how $\Ent(\Psi)$ relates to $\Ent(g_i^{-1}\circ f_i)$.

\begin{figure}[!ht]
	\begin{tikzpicture}[->,>=stealth',auto, scale=2]
	\node (1) at (0,1) {$X_0$};
	\node (2) at (0,0) {$X_0$};
	\node (3) at (1,1) {$X_1$};
	\node (4) at (1,0) {$X_1$};
	\node (5) at (2,1) {$X_2$};
	\node (6) at (2,0) {$X_2$};
	\node (7) at (3,1) {$X_3$};
	\node (8) at (3,0) {$X_3$};
	\node (9) at (4,1) {$\ldots$};
	\node (10) at (4,0) {$\ldots$};
	\draw [->] (3) to (1);
	\draw [->] (5) to (3);
	\draw [->] (7) to (5);
	\draw [->] (9) to (7);
	\draw [->] (4) to (2);
	\draw [->] (6) to (4);
	\draw [->] (8) to (6);
	\draw [->] (10) to (8);
	
	\draw[->, dashed] (1) to (2);
	\draw[->, dashed] (3) to (4);
	\draw[->, dashed] (5) to (6);
	\draw[->, dashed] (7) to (8);

	\node at (0.55,1.15) {\tiny $f_1$};
	\node at (1.55,1.15) {\tiny $f_2$};
	\node at (2.55,1.15) {\tiny $f_3$};
	\node at (3.55,1.15) {\tiny $f_4$};
	\node at (0.55,-0.2) {\tiny $f_1$};
	\node at (1.55,-0.2) {\tiny $f_2$};
	\node at (2.55,-0.2) {\tiny $f_3$};
	\node at (3.55,-0.2) {\tiny $f_4$};
	
	
	\node at (0.35,0.5) {\tiny $g_1\circ f_1^{-1}$};
	\node at (1.35,0.5) {\tiny $g_2\circ f_2^{-1}$};
	\node at (2.35,0.5) {\tiny $g_3\circ f_3^{-1}$};
	\node at (3.35,0.5) {\tiny $g_4\circ f_4^{-1}$};
	\end{tikzpicture}
	\caption{Recovering straight-down components of a diagonal map.}
	\label{fig:diagonalStraightDown}
\end{figure}
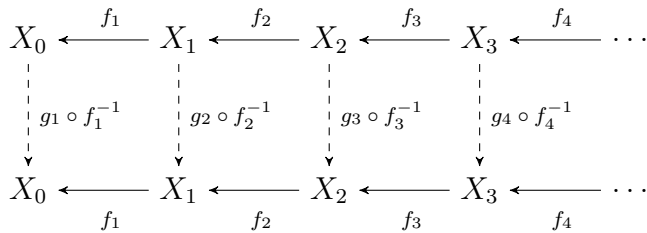

\begin{proposition}\label{prop:entdiag}
	Assume $f_i,g_i\colon X_{i+1}\to X_i$ are continuous for $i\in\N$, and assume they satisfy:
	\begin{itemize}
		\item[(i)] $g_i\circ f_{i+1}=f_i\circ g_{i+1}$, and
		\item[(ii)] $g_{i+1}\circ f_{i+1}^{-1}=f_i^{-1}\circ g_i$, $i\in\N$.
	\end{itemize}
	Let $X=\underleftarrow{\lim}(X_i,f_i)$ and  $\Psi\colon X\to X$ be the diagonal map $$\Psi((x_0,x_1,x_2,\ldots))=(g_1(x_1),g_2(x_2),g_3(x_3),\ldots).$$
	Then $\Ent(g_i\circ f_{i}^{-1})\leq \Ent(\Psi)$ for every $i\in\N$. In particular, it follows that $$\Ent(\Psi)=\lim_{i\to\infty}\Ent(g_i\circ f_i^{-1}).$$
\end{proposition}
\begin{proof}
	Let $i\in\N$, $n\in\N$, and let $(x_1,x_2,\ldots,x_n)\in Orb_n(g_i\circ f_i^{-1})$ (specifically, $x_k\in X_{i-1}$ for every $1\leq k\leq n$). We will first show that there exists $\xi\in\underleftarrow{\lim}(X_i,f_i)$ such that $\pi_i(\Psi^k(\xi))=x_{k+1}$, for $0\leq k<n$. Since $x_{k+1}\in g_i\circ f_i^{-1}(x_k)$ for $1\leq k<n$, there exists $y_k^1\in X_{i+1}$ such that $f_i(y_k^1)=x_k$, and $g_i(y_k^1)=x_{k+1}$, see Figure~\ref{fig:construction}. Furthermore, we have $y_{k+1}^1\in f_i^{-1}\circ g_i(y_k^1)=g_{i+1}\circ f_{i+1}^{-1}(y_k^1)$, so there is $y_k^2\in X_{i+2}$ such that $f_{i+1}(y_k^2)=y_k^1$, and $g_{i+1}(y_k^2)=y_{k+1}^2$, see Figure~\ref{fig:construction}. We continue inductively and find $$\xi=(f_1\circ\ldots\circ f_{i-1}(x_1),\ldots, f_{i-1}(x_1), x_1, y_1^1, y_1^2, \ldots, y_1^{n-1}, \ldots)\in\underleftarrow{\lim}(X_i,f_i).$$ 
	Note that $\pi_i(\xi)=x_1$, and $\pi_i(\Psi^k(\xi))=g_i\circ g_{i+1}\circ\ldots\circ g_{i+k-1}(y_1^k)=x_{k+1}$ for all $1\leq k<n$.
	
	Now let $\{(x_1^1,x_2^1,\ldots x_n^1), (x_1^2,x_2^2,\ldots x_n^2), \ldots, (x_1^k,x_2^k,\ldots x_n^k)\}\subset Orb_n(g_i\circ f_i^{-1})$ be an $(n,\eps)$-separated set. Let $\{\xi_1, \xi_2, \ldots, \xi_k\}\subset\underleftarrow{\lim}(X_i,f_i)$ be a corresponding set of points in the inverse limit as constructed above. Note that
	$\{(\xi_1, \Psi(\xi_1), \ldots, \Psi^{n-1}(\xi_1)),\\ (\xi_2, \Psi(\xi_2), \ldots, \Psi^{n-1}(\xi_2)), \ldots, (\xi_k, \Psi(\xi_k), \ldots, \Psi^{n-1}(\xi_k))\}\subset Orb_n(\Psi)$ is then an $(n,\eps/2^i)$-separated set of $\Psi$. Thus $s_{n,\eps}(g_i\circ f_i^{-1})\leq s_{n,\eps/2^i}(\Psi)$, which implies $\Ent(g_i\circ f_i^{-1})\leq \Ent(\Psi)$. Theorem~\ref{thm:liminf} finishes the proof.
\end{proof}

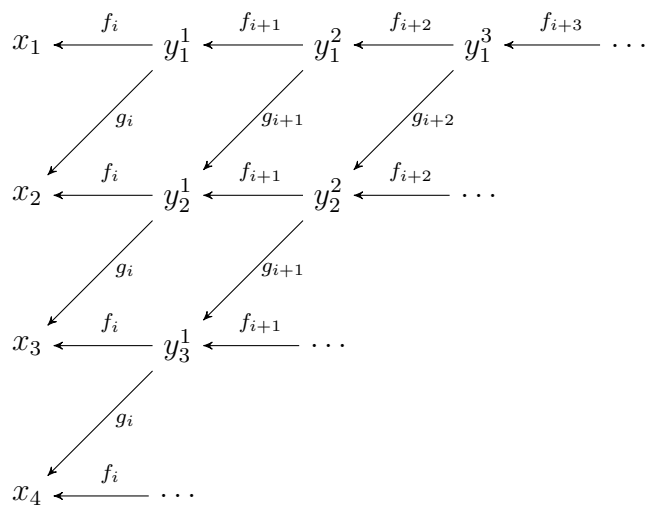
\begin{figure}[!ht]
	\begin{tikzpicture}[->,>=stealth',auto, scale=2]
	\node (13) at (0,0) {$x_4$};
	\node (14) at (1,0) {$\ldots$};
	
	\node (10) at (0,1) {$x_3$};
	\node (11) at (1,1) {$y_3^1$};
	\node (12) at (2,1) {$\ldots$};
	
	\node (6) at (0,2) {$x_2$};
	\node (7) at (1,2) {$y_2^1$};
	\node (8) at (2,2) {$y_2^2$};
	\node (9) at (3,2) {$\ldots$};
	
	\node (1) at (0,3) {$x_1$};
	\node (2) at (1,3) {$y_1^1$};
	\node (3) at (2,3) {$y_1^2$};
	\node (4) at (3,3) {$y_1^3$};
	\node (5) at (4,3) {$\ldots$};
	
	\draw[->] (2)--(1);
	\draw[->] (3)--(2);
	\draw[->] (4)--(3);
	\draw[->] (5)--(4);
	\draw[->] (2)--(6);
	\draw[->] (3)--(7);
	\draw[->] (4)--(8);
	\draw[->] (7)--(6);
	\draw[->] (8)--(7);
	\draw[->] (9)--(8);
	\draw[->] (7)--(10);
	\draw[->] (8)--(11);
	\draw[->] (11)--(10);
	\draw[->] (12)--(11);
	\draw[->] (11)--(13);
	\draw[->] (14)--(13);

	\node at (0.55,3.15) {\tiny $f_i$};
	\node at (1.55,3.15) {\tiny $f_{i+1}$};
	\node at (2.55,3.15) {\tiny $f_{i+2}$};
	\node at (3.55,3.15) {\tiny $f_{i+3}$};
	\node at (0.55,2.15) {\tiny $f_i$};
	\node at (1.55,2.15) {\tiny $f_{i+1}$};
	\node at (2.55,2.15) {\tiny $f_{i+2}$};
	\node at (0.55,1.15) {\tiny $f_i$};
	\node at (1.55,1.15) {\tiny $f_{i+1}$};
	\node at (0.55,0.15) {\tiny $f_i$};
	
	\node at (0.65,2.5) {\tiny $g_i$};
	\node at (1.7,2.5) {\tiny $g_{i+1}$};
	\node at (2.7,2.5) {\tiny $g_{i+2}$};
	\node at (0.65,1.5) {\tiny $g_i$};
	\node at (1.7,1.5) {\tiny $g_{i+1}$};
	\node at (0.65,0.5) {\tiny $g_i$};
	\end{tikzpicture}
	\caption{Construction of $y_k^j\in X_{i+j}$ in the proof of Proposition~\ref{prop:entdiag}.}
	\label{fig:construction}
\end{figure} 

In particular, if we are given two commutative maps $f,g\colon I\to I$, such that $g\circ f^{-1}=f^{-1}\circ g$, then for the diagonal map $\Psi\colon\underleftarrow{\lim}(I,f)\to\underleftarrow{\lim}(I,f)$, given by $\psi((x_0,x_1,x_2,\ldots))=(g(x_1),g(x_2),\ldots)$, we have
$$\Ent(\Psi)=\Ent(g\circ f^{-1}).$$

\section{Entropy of $F_n\circ F_m^{-1}=F_m^{-1}\circ F_n$}\label{sec:main}

For $n\geq 2$ denote by $F_n\colon I\to I$ an open map with $n-1$ critical points. Note that $\Ent(F_n)=\log(n)$ (see \cite{MisSl}). In this section we prove that if $F_n\circ F_m^{-1}=F_m^{-1}\circ F_n$, then $\Ent(F_n\circ F_m^{-1})=\max\{\log(n),\log(m)\}=\max\{\Ent(F_n),\Ent(F_m)\}$.

For example, denote by $T_n$ the symmetric tent map with $n-1$ critical points. That is, define $T_n(i/n)=0$ if $i\in\{0,\ldots,n\}$ is even, and $T_n(i/n)=1$ if $i\in\{0,\ldots,n\}$ is odd, and extend linearly. In \cite[Propositions~3.3 and 3.4]{AM-stronglycomm} we show that $T_m^{-1}\circ T_n=T_n\circ T_m^{-1}$ if and only if $n$ and $m$ are relatively prime. See the graphs of some $T_m^{-1}\circ T_n$ in Figure~\ref{fig:tmtn}.

\begin{figure}[!ht]
	\centering
	\begin{tikzpicture}[scale=0.5]
	\draw[thick] (0,0)--(0,8)--(8,8)--(8,0)--(0,0);
	\draw[thick] (0,0)--(8,16/3)--(4,8)--(0,16/3)--(8,0);
	\node at (4,9) {\small $T_3^{-1}\circ T_2=T_2\circ T_3^{-1}$};
	\end{tikzpicture}
	\begin{tikzpicture}[scale=0.5]
	\draw[thick] (0,0)--(0,8)--(8,8)--(8,0)--(0,0);
	\draw[thick] (0,0)--(8,24/5)--(8/3,8)--(0,32/5)--(8,8/5)--(16/3,0)--(0,16/5)--(8,8);
	\node at (4,9) {\small $T_5^{-1}\circ T_3=T_3\circ T_5^{-1}$};
	\end{tikzpicture}
	\begin{tikzpicture}[scale=0.5]
	\draw[thick] (0,0)--(0,8)--(8,8)--(8,0)--(0,0);
	\draw[thick] (0,0)--(16/3,8)--(8,4)--(16/3,0)--(0,8);
	\draw[thick] (8,0)--(8/3,8)--(0,4)--(8/3,0)--(8,8);
	\node at (4,9) {\small $T_4^{-1}\circ T_6\neq T_6\circ T_4^{-1}$};
	\end{tikzpicture}
	\caption{Graphs of some $T_m^{-1}\circ T_n$, where $T_n$ denotes the symmetric tent map with $n-1$ critical points. Note that $T_4^{-1}\circ T_6\neq T_6\circ T_4^{-1}=T_3\circ T_2^{-1}$.}
	\label{fig:tmtn}
\end{figure}
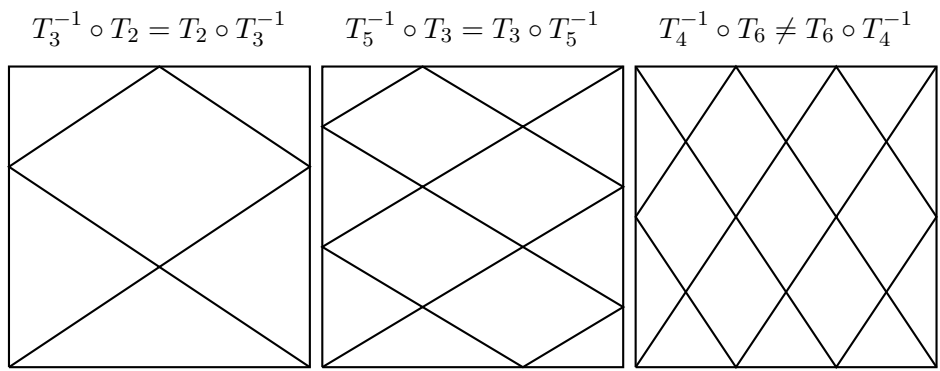

\begin{definition}
	Let $\Psi\colon X\to X$ be a set valued map and $N\in\N$. We say that $\Psi$ has an {\em $N$-horseshoe} if there exist disjoint closed sets $A_1, A_2, \ldots, A_N\subset X$ such that 
	$\cup_{i=1}^NA_i\subset \cap_{i=1}^N\Psi(A_i)$.
\end{definition}

\begin{proposition}\label{prop:horseshoes}
	If a set valued map $\Psi\colon X\to X$ has an $N$-horseshoe, then $\Ent(\Psi)\geq \log(N).$
\end{proposition}
\begin{proof} Let $A_1$, $A_2$, $\ldots, A_N$ be an $N$-horseshoe for $\Psi$. For $i_1, i_2\in\{1, \ldots N\}$ define $A_{i_1,i_2}=\{x\in A_{i_1}: \Psi(x)\cap A_{i_2}\not = \emptyset\}$. Since $A_{i_2}\subset \Psi(A_{i_1})$, it follows that $A_{i_1,i_2}\neq\emptyset$, and also $A_{i_2}\subset \Psi(A_{i_1,i_2})$. 
	
	Continuing inductively, let $n\geq 1$ and suppose that $A_{i_1,...,i_n}$ has been found such that $A_{i_n}\subset\Psi^{n-1}(A_{i_1,...,i_n})$. For $i_{n+1}\in\{1, \ldots, N\}$ define $A_{i_1,...,i_n,i_{n+1}}=\{x\in A_{i_1,...,i_n}: \Psi^{n}(x)\cap A_{i_{n+1}}\not = \emptyset\}$. Since $A_{i_{n+1}}\subset \Psi(A_{i_n})\subset \Psi^{n}(A_{i_1,...,i_n})$, it follows that $A_{i_1,\ldots, i_n}\neq\emptyset$, and also $A_{i_{n+1}}\subset \Psi^n(A_{i_1,...,i_{n+1}})$.
	
	Fix $n\in\N$ and let $\eps<\min\{d_X(A_i,A_j): i, j\in\{1, \ldots, N\}, i\neq j\}$, where $d_X$ denotes the metric on $X$. Note that for every $(i_1, \ldots, i_n)\in\{1, \ldots, N\}^n$, we can find $(x_1, \ldots, x_n)\in Orb_n(\Psi)$ such that $x_j\in A_{i_j}$ for every $j\in\{1, \ldots, n\}$. We just need to pick $x_1\in A_{i_1, \ldots, i_n}\neq\emptyset$. Then since $x_1\in A_{i_1}\cap A_{i_1, i_2}\cap \ldots\cap A_{i_1, \ldots, i_{n}}$, we know that $\Psi^j(x_1)\cap A_{i_{j+1}}\neq\emptyset$ for every $j\in\{1, \ldots, n-1\}$, so we can pick $x_{j+1}\in \Psi^j(x_1)\cap A_{i_{j+1}}$. Denote the constructed orbit $(x_1, \ldots, x_n)\in Orb_n(\Psi)$ by $x(i_1, \ldots i_n)$ and note that $S=\{x(i_1, \ldots, i_n): (i_1, \ldots, i_n)\in\{1, \ldots, N\}^n\}$ is $(n, \eps)$-separated. So $s_{n, \eps}(\Psi)\geq N^n$ and we conclude $\Ent(\Psi)=\lim_{\eps\to 0}\limsup_{n\to\infty}\frac{1}{n}\log(s_{n,\eps}(\Psi))\geq \log(N).$
\end{proof}

Note that for single-valued maps $f$ it holds $\Ent(f^n)=n\Ent(f)$ for every $n\in\N$. However, if $F$ is set-valued, that is no longer be the case, see the following example.

\begin{example}
	Denote by $T_2\colon I\to I$ the symmetric tent map, \ie $T_2(x)=\min\{2x, 2(1-x)\}$ for all $x\in I$. Define set-valued map $F\colon I\to I$ as $F(x)=T_2^{-1}\circ T_2(x)$ for all $x\in I$. Define $A_1=[0,1/3]$ and $A_2=[2/3,1]$ and note that $F(A_1)=F(A_2)=A_1\cup A_2$, thus $A_1, A_2$ is a $2$-horseshoe for $F$ and it follows that $\Ent(F)\geq \log 2$. Since $F^n=F$ for all $n\in\N$ it follows that $\Ent(F^n)=\Ent(F)<n\Ent(F)$ for all $n\geq 2$.
\end{example}

However, we have the following relation between $\Ent(F)$ and $\Ent(F^n)$:

\begin{proposition}\label{prop:entIter}
	Let $F\colon X\to X$ be a set-valued function. Then
	$$\Ent(F)\geq \sup_{n\in\N} \frac {\Ent(F^n)}{n}.$$
\end{proposition}
\begin{proof}
	Fix $k\in\N$ and note that for any $(n,\eps)$-separated set for $F^k$ there exists an $(nk,\eps)$-separated set for $F$, simply extend an
	 $n$-orbit of $F^k$ to $nk$-orbit of $F$. It follows that $s_{n,\eps}(F^k)\leq s_{nk,\eps}(F)$. So we have 
	\begin{equation*}
	\begin{split}
	\Ent(F)&=\lim_{\eps\to 0}\limsup_{n\to\infty} \frac{\log(s_{n,\eps}(F))}{n}\geq \lim_{\eps\to 0}\limsup_{n\to\infty}\frac{\log(s_{nk,\eps}(F))}{nk}\\ 
	&\geq \lim_{\eps\to 0}\limsup_{n\to\infty}\frac{\log(s_{n,\eps}(F^k))}{nk}=\frac 1k\Ent(F^k),
	\end{split}
	\end{equation*}
	for every $k\in\N$.
\end{proof}

\begin{proposition}\label{prop:lowerbound}
	Let $F_n\circ F_m^{-1}=F_m^{-1}\circ F_n$. Then $\Ent(F_n\circ F_m^{-1})\geq\max\{\log n, \log m\}$.
\end{proposition}
\begin{proof}
	Denote by $A_i$, $1\leq i\leq n$, intervals in $I$ such that $F_n|_{A_i}$ are monotone, and $F_n(A_i)=I$, for all $i\in\{1,\ldots, n\}$. Furthermore, order the intervals such that if $i<j$, then $x\leq y$ for all $x\in A_i, y\in A_{j}$.
	
	Note that $F_m^{-1}\circ F_n(A_i)=I$ for every $1\leq i\leq n$. So $\{A_i: 1\leq i\leq n, i \text{ odd}\}$ is $\lfloor \frac{n+1}{2}\rfloor$-horseshoe for $F_m^{-1}\circ F_n$, so Proposition~\ref{prop:horseshoes} implies that $\Ent(F_m^{-1}\circ F_n)\geq \log(\lfloor \frac{n+1}{2}\rfloor)$. Denote for simplicity $F:=F_m^{-1}\circ F_n$ and note that, since $F_n\circ F_m^{-1}=F_m^{-1}\circ F_n$, we have $F^k=(F_m^{-1}\circ F_n)^k=F_m^{-k}\circ F_n^k=F_{m^k}^{-1}\circ F_{n^k}$ for every $k\in\N$. We conclude that $\Ent(F^k)\geq \log(\lfloor \frac{n^k+1}{2}\rfloor)$, for every $k\in\N$. Proposition~\ref{prop:entIter} then implies that for every $k\in\N$
	\begin{equation*}
	\begin{split}
	\Ent(F)&\geq \frac 1k\Ent(F^k)\geq \frac 1k\log\left(\left\lfloor \frac{n^k+1}{2}\right\rfloor\right)\geq \frac 1k\log\left( \frac{n^k}{2}\right)=\log(n)-\frac{\log(2)}{k}.
	\end{split}
	\end{equation*}
	We conclude that $\Ent(F)\geq\log(n)$. Furthermore, since $\Ent(F)=\Ent(F^{-1})$ (see \cite{KelTen}), it follows that $\Ent(F)\geq\log(m)$ also.
\end{proof}

Let $n>m$, $F_n\circ F_m^{-1}=F_m^{-1}\circ F_n$. Let $\Gamma=\Gamma(F_m^{-1}\circ F_n)=\{(F_m(t),F_n(t)): t\in [0,1]\}$. Then it follows that $\Gamma^k=\{(F^k_m(t),F^k_n(t)): t\in [0,1]\}$. It will be easier to represent $\Gamma^k$ at times by its parameterization. So we denote $[p,q]_k^P=\{(F^k_m(t),F^k_n(t)): t\in [p,q]\}$. Let $\pi_1$ and $\pi_2$ be projections on the first and second coordinates of $\Gamma$, and $\pi_1^k$ and $\pi_2^k$ be projections on the first and second coordinates of $\Gamma^k$.   Arc $A\subset \Gamma^k$ is a {\it monotone arc} of $\Gamma^k$ if  $\pi_1^k|_{A}$ and $\pi_2^k|_{A}$ are homeomorphisms. $A$ is a {\it monotone branch} if it is a maximal monotone arc.

The following result is found in \cite{AM-stronglycomm}:

\begin{lemma}[Corollary~4.15 in \cite{AM-stronglycomm}]\label{lem:commoncritpt}
	If $f,g\colon I\to I$ are piecewise monotone maps such that $f\circ g^{-1}=g^{-1}\circ f$, then $c$ is a critical point  of $f$ if and only if it is not a critical point of $g$.
\end{lemma}

Let $0<t_1<t_2<\ldots<t_{n-1}<1$ be the critical points of $F_n$ and $0<\widehat{t}_1<\widehat{t}_2<\ldots<\widehat{t}_{m-1}<1$ be the critical points of $F_m$. Then it follows from Lemma~\ref{lem:commoncritpt} that $\{t_i\}_{i=1}^{n-1}\cap\{\hat t_i\}_{i=1}^{m-1}=\emptyset$. For ease of notation let $t_0=\widehat{t}_0=0=s_0$, $t_n=\widehat{t}_m=1=s_{n+m-1}$ and $\{s_i\}_{i=0}^{n+m-1}=\{t_i\}_{i=0}^{n}\cup\{\widehat{t}_i\}_{i=0}^{m}$ where $0=s_0<s_1<s_2<...<s_{m+n-1}=1$.  Note that for each $i$, $[s_i,s_{i+1}]_1^P$ is a monotone branch of $\Gamma$; hence, $\Gamma$ has $n+m-1$ monotone branches. The reader is encouraged to see Figure~\ref{fig:tmtn} and Figure~\ref{fig:consistent}.

\begin{lemma}\label{lem:crit}
	Let $F_n\circ F_m^{-1}=F_m^{-1}\circ F_n$. Let $0<t_1<\ldots<t_{n-1}<1$ be critical points of $F_n$, and $0<\hat t_1<\ldots<\hat t_{m-1}<1$ be critical points of $F_m$. Then $t_i<\hat t_j<t_{i+1}$ if and only if $i=\lfloor\frac nmj\rfloor$.
\end{lemma}
\begin{proof}
	By Lemma~\ref{lem:commoncritpt}, $t_i\neq \hat t_j$ for every $i\in\{1,\ldots,n-1\}$, and $j\in\{1,\ldots,m-1\}$. Note that every $\hat t_j$ is $(jn)$th critical point of $F_n\circ F_m$, and every $t_i$ is $(im)$th critical point of $F_m\circ F_n=F_n\circ F_m$. Given $j$, let $i_j$ and $k$   be such that $jn=i_jm+k$, where $0<k<m$. so $t_{i_j}<\hat t_j<t_{i_j+1}$. Notice that this is true if and only if $i_j=\frac nmj-\frac km=\lfloor\frac nmj\rfloor$.
\end{proof}

\begin{remark}\label{rem:stuff}
	Note the following:
	\begin{enumerate}
		\item $\pi_2([t_i,t_{i+1}]_1^P)=[0,1]$ and $\pi_1([\widehat{t}_i,\widehat{t}_{i+1}]_1^P)=[0,1]$ for each $i$.
		\item $|\{\widehat{t}_i\}_{i=1}^m\cap [t_j,t_{j+1}]|\leq 1$, since $n>m$, and by Lemma~\ref{lem:crit}.
		\item if $\widehat{t}_i\in [t_j,t_{j+1}]$, then $[t_j,\widehat{t}_{i}]_1^P$ and $[\widehat{t}_i,t_{j+1}]_1^P$ are monotone branches. 
		\item If $\{\widehat{t}_i\}_{i=1}^m\cap [t_j,t_{j+1}]=\emptyset$ then  $[t_j,t_{j+1}]_1^P$ is a monotone branch.
	\end{enumerate}
\end{remark} 

An arc $[p,q]^P_{k}\subset \Gamma^{k}$ is a {\it consistent monotone arc} if there exist monotone arcs $[p_1,q_1]_1^P$, $[p_2,q_2]_1^P,\ldots, [p_{k},q_{k}]_1^P$ such that  
$\pi_1([p_1,q_1]_1^P)=\pi_1^k([p,q]_k^P)$, $\pi_2([p_i,q_i]_1^P)=\pi_1([p_{i+1},q_{i+1}]_1^P)$ for $1\leq i\leq k-1$ and 
$\pi_2([p_k,q_k]_1^P)=\pi_2^k([p,q]_k^P)$. 

Let $\mathcal{M}_1$ be the collection of monotone branches of $\Gamma$. Suppose $\mathcal{M}_k$ has been defined for $\Gamma^k$. Suppose that $A\in \mathcal{M}_1$ and $B\in \mathcal{M}_k$ are such that $\pi_2(A)\cap \pi^k_1(B)=[z_1,z_2]=Z$ where $z_1<z_2$. Define $C_k(A,B)=\{(t,\pi_2^k\circ(\pi^k_1|_B)^{-1}\circ \pi_2\circ (\pi_1|_A)^{-1}(t)): t\in \pi_1\circ (\pi_2|_A)^{-1}(Z)\}=\{\pi_1\circ(\pi_2|_A)^{-1}(z),\pi_2^k\circ(\pi_1^k|_B)^{-1}(z): z\in Z\}$. Then define $\mathcal{M}_{k+1}=\{C_k(A,B):  A\in \mathcal{M}_1\mbox{, } B\in \mathcal{M}_k \mbox{ and }  \mbox{Int}(\pi_2(A)\cap \pi^k_1(B))\not = \emptyset\}$. See the following example.

\begin{example}\label{ex:consistent}
	Let $n=3,m=2$, and let $T_n, T_m$ be symmetric tent maps. Then $T_n\circ T_m^{-1}=T_m^{-1}\circ T_n=\Gamma$. Let $A$ denote the straight line in the graph of $\Gamma$ connecting $(0,0)$ and $(2/3,1)$, $B$ be the line connecting $(2/3,1)$ to $(1,1/2)$, $C$ the line connecting $(1,1/2)$ to $(2/3,0)$, and $D$ the line connecting $(2/3,0)$ to $(0,1)$, see Figure~\ref{fig:consistent}. Then $A,B,C,D$ are monotone branches of $\Gamma$, and also $\mathcal{M}_1=\{A,B,C,D\}$. 
	
	Note that every monotone branch of $\Gamma$ can be understood as a graph of a homeomorphism, \eg $A$ can be understood as a graph of $h_A\colon \pi_1(A)=[0,2/3]\to\pi_2(A)=I$, and $B$ can be understood as a graph of $h_B\colon\pi_1(B)\to\pi_2(B)$. Then $C_1(A,B)$ will be the graph of $h_A\circ h_B|_J$, where $J$ is the maximal interval in $I$ such that $h_A\circ h_B|_J$ is monotone, and  $C_1(A,B)$ is a monotone arc of $\Gamma^2$. One easily checks that $J=[4/9,2/3]$, and thus $C_1(A,B)$ is a graph of $J\xrightarrow{h_A}[2/3,1]\xrightarrow{h_B}[1/2,1]$, see Figure~\ref{fig:consistent}. Similarly, $C_1(C,A)$ is a graph of $h_C\circ h_A|_{J'}$, where $J'\subset I$ is the maximal such that $h_C\circ h_A|_{J'}$ is monotone (if it exists). We can easily check that $J'=[2/3,1]$, thus $C_1(C,A)$ is the graph of $J'\xrightarrow{h_C}[0,1/2]\xrightarrow{h_A}[0,3/4]$. We can similarly check that $\mathcal{M}_2=\{C_1(X,Y): X,Y\in\{A,B,C,D\}, \text{ and if $X=C$, then } Y\in\{A,D\}\}$. Arcs $C_1(C,C)$, and $C_1(C,B)$ are not defined since there are no non-degenerate intervals $J,J'\subset I$ such that $h_C\circ h_C|_J, h_C\circ h_B|_{J'}$ are monotone.
	
	We proceed similarly to find arcs in $\mathcal{M}_3$. Note that  $C_2(A,C_1(A,B))$ is the graph of $h_A\circ(h_A\circ h_B)|_J$, where $J\subset I$ is the maximal such that $h_A\circ(h_A\circ h_B)|_J$ is monotone. One easily finds $J=[8/27,4/9]$, and thus $C_2(A,C_1(A,B))$ is the graph of $J\xrightarrow{h_A}[4/9,2/3]\xrightarrow{h_A}[2/3,1]\xrightarrow{h_B}[1/2,1]$, which is  a monotone arc in $\Gamma^3$.
\end{example}

\begin{figure}
	\centering
	\begin{tikzpicture}[scale=5]
	\draw (0,0)--(0,1)--(1,1)--(1,0)--(0,0);
	\draw (0,0)--(2/3,1)--(1,1/2)--(2/3,0)--(0,1);
	\draw[fill] (0,0) circle (0.02);
	\draw[fill] (2/3,1) circle (0.02);
	\draw[fill] (1,1/2) circle (0.02);
	\draw[fill] (2/3,0) circle (0.02);
	\draw[fill] (0,1) circle (0.02);
	\node at (0.22,1/4) {$A$};
	\node at (0.75,3/4) {$B$};
	\node at (0.75,1/4) {$C$};
	\node at (0.22,3/4) {$D$};
	\end{tikzpicture}
	\hspace{10pt}
	\begin{tikzpicture}[scale=5]
	\draw (0,0)--(0,1)--(1,1)--(1,0)--(0,0);
	\draw (0,0)--(4/9,1)--(8/9,0)--(1,1/4)--(6/9,1)--(2/9,0)--(0,1/2)--(2/9,1)--(6/9,0)--(1,3/4)--(8/9,1)--(4/9,0)--(0,1);
	\draw[very thick] (4/9,1)--(2/3,1/2);
	\node[above] at (0.55,1) {\small $C_1(A,B)$};
	\draw[->] (0.6,1.03)--(0.5,0.9);
	\draw[very thick] (2/3,0)--(1,3/4);
	\node[right] at (1,0.5) {\small $C_1(C,A)$};
	\draw[->] (1.1,0.46)--(5/6,0.33);
	
	\node at (0,-0.002) {};
	\end{tikzpicture}
	\caption{Graph of $\Gamma=T_3\circ T_2^{-1}=T_2^{-1}\circ T_3$ from Example~\ref{ex:consistent} with monotone branches $A,B,C,D$ (left), and the graph of $\Gamma^2$ with consistent monotone arcs $C_1(A,B)$, and $C_1(C,A)$ in bold (right).}
	\label{fig:consistent}
\end{figure}
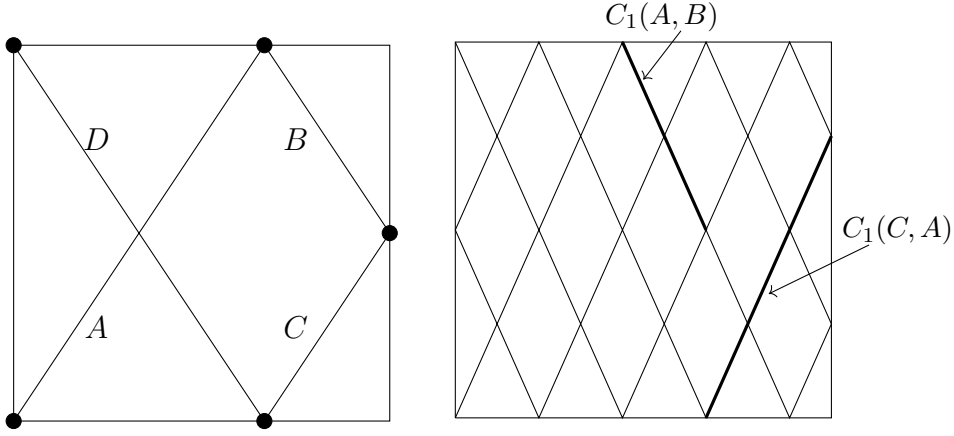

\begin{proposition}\label{FEnt1}  $\mathcal{M}_k$ has the following properties for each $k$:
	\begin{enumerate}
		\item $C$ is a consistent monotone arc for each $C\in \mathcal{M}_k$.
		\item $\Gamma^k=\bigcup_{C\in \mathcal{M}_k}C$.
	\end{enumerate}
\end{proposition}
\begin{proof} Since $\Gamma$ is the union of its monotone branches and each monotone branch is trivially a consistent monotone arc, (1) and (2) are satisfied for $k=1$. 
	
	Continuing inductively, suppose that $\mathcal{M}_k$ satisfies (1) and (2). Let $C_k(A,B)\in \mathcal{M}_{k+1}$ and $\pi_2(A)\cap \pi^k_1(B)=Z$.  Since $B\in\mathcal{M}_k$, and $\mathcal{M}_k$ satisfies (1), it follows that $B$ is a consistent monotone arc. So by the definition there exist monotone arcs $B_1, B_2,...,B_k\subset \Gamma$ such that  
	$\pi_1(B_1)=\pi_1^k(B)$, $\pi_2(B_i)=\pi_1(B_{i+1})$ for $1\leq i\leq k-1$ and 
	$\pi_2(B_k)=\pi^k_2(B)$. Let $A_1=(\pi_2|_A)^{-1}(Z)$, $A_2= (\pi_1|_{B_1})^{-1}(Z)$, and for $2\leq k$, let $A_{i+1}=(\pi_1|_{B_i})^{-1}\circ \pi_2(A_i)$. Then $\pi_1(A_1)=\pi_1(C_k(A,B))$, $\pi_2(A_i)=\pi_1(A_{i+1})$ for $1\leq i\leq k$, and $\pi_2(A_{k+1})=\pi_2(C_k(A,B))$, so $C_k(A,B)$ is a consistent monotone arc of $\Gamma^{k+1}$.
	
	Let $(x,y)\in \Gamma^{k+1}$. Then there exists $z\in [0,1]$ such that $(x,z)\in \Gamma$ and $(z,y)\in \Gamma^k$. Let $B\in \mathcal{M}_k$ such that $(z,y)\in B$. Without loss of generality, assume there exists $z_1\in \pi^k_1(B)$ such that $z<z_1$. Then there  exists a monotone branch $A\in \mathcal{M}_1$ such that $(x,z)\in A$ and $z_2\in \pi_2(A)$ where $z<z_2$. Hence $\mbox{Int}(\pi_2(A)\cap\pi_1^k(B))\not=\emptyset$. So $(x,y)\in C_k(A,B)\in \mathcal{M}_{k+1}$. Hence $\mathcal{M}_{k+1}$ satisfies (2).
\end{proof}

\begin{lemma}\label{FEnt2} If $C_1, C_2\in \mathcal{M}_k$ and $\mathrm{Int}(C_1\cap C_2)\not=\emptyset$, then $C_1=C_2$.
\end{lemma}
\begin{proof} Proof is by induction. This is trivial for $\mathcal{M}_1$. Suppose  there exist $C_1, C_2\in \mathcal{M}_{k+1}$ such that $\mbox{Int}(C_1\cap C_2)\not=\emptyset$. Let $(x,y)\in\mbox{Int}(C_1\cap C_2)$, $A_1, A_2\in \mathcal{M}_1$ and  $B_1, B_2\in \mathcal{M}_k$ such that $C_1=C_k(A_1,B_1)$ and $C_2=C_k(A_2,B_2)$. Let $Z_1=\pi_2(A_1)\cap \pi_1^k(B_1)$ and $Z_2=\pi_2(A_2)\cap \pi_1^k(B_2)$. Since $\pi_1^k|_{B_1}$ and $\pi_1^k|_{B_2}$ are homeomorphisms onto $Z_1$ and $Z_2$ respectively, it follows that there  exist $z\in \mbox{Int}(Z_1\cap Z_2)$ such that $(x,z)\in \Gamma$ and $(z,y)\in \Gamma^k$. Thus, $(x,z)\in \mbox{Int}(A_1\cap A_2)$ and $(z,y)\in \mbox{Int}(B_1\cap B_2)$. Therefore, by the induction hypothesis, $A_1=A_2$ and $B_1=B_2$. It follows that $C_1=C_2$.
\end{proof}

\begin{lemma}\label{FEnt3}
	Let $A$ be a monotone arc of $\Gamma\subset I\times I$ and $B$ be a monotone arc of $\Delta\subset I\times I$ such that $\pi_1(A)=[x_1,x_2]$, $\pi_2(A)=[y_1,y_2]={\pi}_1(B)$ and ${\pi}_2(B)=[z_1,z_2]$. Then there exists a monotone arc $C\subset \Delta\circ \Gamma$ such that ${\pi}_1(C)=[x_1,x_2]$ and ${\pi}_2(C)=[z_1,z_2]$, where $\pi_1,\pi_2$ are the respective projection maps of $I\times I$ to $I$.
\end{lemma}
\begin{proof} Let $\alpha:[x_1,x_2]\to [y_1,y_2]$ be defined by $\alpha=\pi_2|_A\circ(\pi_1|_A)^{-1}$ and $\beta:[y_1,y_2]\to [z_1,z_2]$ be defined by $\beta={\pi}_2|_B\circ({\pi}_1|_B)^{-1}$. Since $A$ and $B$ are monotone arcs, $\alpha$ and $\beta$ are homeomorphisms. Define $C=\{(t,\beta\circ\alpha(t)): t\in [x_1,x_2]\}$. Since $(t,\alpha(t))\in \Gamma$ and  $(\alpha(t),\beta\circ\alpha(t))\in \Delta$ it follows that $C\subset \Delta\circ \Gamma$. Since $\beta\circ\alpha:[x_1,x_2]\to[z_1,z_2]$ is a homeomorphism, it follows that $C$ is a monotone arc.
\end{proof} 

For $A\in \mathcal{M}_1$, define $\mathcal{C}^k(A)=\{C_k(A,B): B\in \mathcal{M}_k \mbox{ and } \mbox{Int}(\pi_2(A)\cap\pi^k_1(B))\not=\emptyset\}$. Then $\mathcal{M}_{k+1}=\bigcup_{A\in\mathcal{M}_1}\mathcal{C}^k(A)$.

\begin{lemma}\label{FEnt4} If $A\in \mathcal{M}_1$ is such that $\pi_2(A)=[0,1]$, then $|\mathcal{C}^k(A)|=|\mathcal{M}_k|$.
\end{lemma}
\begin{proof} If $B\in \mathcal{M}_k$, then \[\mbox{Int}(\pi^k_1(B)\cap\pi_2(A))=\mbox{Int}(\pi^k_1(B)\cap [0,1])=\mbox{Int}(\pi^k_1(B))\not=\emptyset.\] Hence, $C(A,B)\in \mathcal{C}^k(A)$.
\end{proof}

\begin{lemma}\label{FEnt5}
	If $[s_1,s_2]_1^P,[s_2,s_3]_1^P\in \mathcal{M}_1$ such that $\pi_2([s_1,s_2]_1^P)=[0,z]$ and $\pi_2([s_2,s_3]_1^P)=[z,1]$ (or vice versa), where $0<z<1$, then $|\mathcal{C}^k([s_1,s_2]_1^P)|+|\mathcal{C}^k([s_2,s_3]_1^P)|\leq|\mathcal{M}_k|+|(\pi^k_1)^{-1}(z)|$.
\end{lemma}
\begin{proof}
	Partition $\mathcal{M}_k$ into the following collections $\mathcal{P}_0=\{B\in\mathcal{M}_k: \pi_1^k(B)\subset [0,z]\}$, $\mathcal{P}_1=\{B\in\mathcal{M}_k: \pi_1^k(B)\subset [z,1]\}$ and $\mathcal{P}_{\mbox{both}}=\{B\in\mathcal{M}_k: z\in \mbox{Int}(\pi_1^k(B))\}$. Note that by Lemma \ref{FEnt2}, $| \mathcal{P}_{\mbox{both}}|\leq |(\pi^k_1)^{-1}(z)|$. So, it follows that  $|\mathcal{C}^k([s_1,s_2]_1^P)|+|\mathcal{C}^k([s_2,s_3]_1^P)|=| \mathcal{P}_{1}|+| \mathcal{P}_{\mbox{both}}|+| \mathcal{P}_{2}|+| \mathcal{P}_{\mbox{both}}|=|\mathcal{M}_k|+| \mathcal{P}_{\mbox{both}}|\leq|\mathcal{M}_k|+|(\pi^k_1)^{-1}(z)|$.
\end{proof}

\begin{lemma}\label{FEnt6} Suppose $n>m$ and $F_n\circ F_m^{-1}=F_m^{-1}\circ F_n=\Gamma$. Then $|\mathcal{M}_k|\leq(k+1)n^k$.
\end{lemma}
\begin{proof} Let $\{\mathcal{P}_{\mbox{full}},\mathcal{P}_{\mbox{split}}\}$ be a partition of $\{[t_i,t_{i+1}]_1^P\}_{i=0}^{n-1}$ defined by \[\mathcal{P}_{\mbox{full}}=\{[t_i,t_{i+1}]_1^P:\mbox{ there are no critical points of } F_m \mbox{ in } [t_i, t_{i+1}]\}=\{A_{j}\}_{j=1}^{n+1-m},\]  \[\mathcal{P}_{\mbox{split}}=\{[t_i,t_{i+1}]_1^P: \mbox{ there is a (unique) critical point of } F_m \mbox{ in } (t_i, t_{i+1})\}=\{\widehat{A}_{d}\}_{d=1}^{m-1}.\]  Recall Remark~\ref{rem:stuff}, and note $|\mathcal{P}_{\mbox{split}}|=$ the number of critical points of $F_m$. Furthermore, if $\widehat{t}_d$ is the critical point of $F_m$ such that $\widehat{A}_d=[t_{i(d)},\widehat{t}_d]_1^P\cup[\widehat{t}_d, t_{i(d)_1}]_1^P$, then define $\widehat{A}^0_d=[t_{i(d)},\widehat{t}_d]^P_1$ and $\widehat{A}^1_d=[\widehat{t}_d, t_{i(d)_1}]_1^P$.\\
	
	So it follows that $|\mathcal{M}_1|=|\mathcal{P}_{\mbox{full}}|+2|\mathcal{P}_{\mbox{split}}|=n+1-m+2(m-1)=n+m-1\leq 2n$.
	
	Suppose that $|\mathcal{M}_k|\leq (k+1)n^k$. Note that it follows from Lemmas \ref{FEnt4} and \ref{FEnt5}, and the fact that $|(\pi_1^k)^{-1}(z)|\leq km$ for all $z\in [0,1]$ that
	
	\begin{eqnarray*}
		|\mathcal{M}_{k+1}|&\leq&|\mathcal{M}_k||\mathcal{P}_{\mbox{full}}|+|\mathcal{M}_k||\mathcal{P}_{\mbox{split}}|+km\\
		&\leq&(k+1)n^k(n+1-m)+(k+1)n^k(m-1)+km\\
		&\leq&(k+1)n^{k+1}+km\\
		&\leq&(k+2)n^{k+1},
	\end{eqnarray*}
which finishes the proof.
\end{proof}

The following proposition is similar to Lemma~3 in \cite{MisSl}. For $k\in\N$ and $0\leq i\leq k-1$ we define $h_{k-1}^i:=h_{k-1}\circ h_{k-2}\circ...\circ h_{i+1}$, where $h_{k-1}^{k-1}=id$. For a finite open cover $\mathcal{U}$ of $I$ we define $\vee_{i=0}^{k-1}h_{k-1}^i(\mathcal{U}):=\{\bigcap_{i=0}^{k-1}h_{k-1}^i(U_i)\not=\emptyset : U_i\in {\mathcal U}\}.$

\begin{proposition}\label{FEnt7} Let $\mathcal{U}$ be a finite  cover of open intervals of $[0,1]$ with no common endpoints, let $\{[a_i,b_i]\}_{i=1}^k$ be a collection of subintervals of $[0,1]$ and $\{h_i\}_{i=1}^{k-1}$ be a collection of homeomorphisms $h_i:[a_{i},b_{i}]\to [a_{i+1},b_{i+1}]$. Let ${\mathcal B}$ be a subcover of $\vee_{i=0}^{k-1}h_{k-1}^i(\mathcal{U})$  of minimal cardinality that covers $[a_k,b_k]$. Then $|\mathcal{B}|\leq k^2|\mathcal{U}|^2$.
\end{proposition}
\begin{proof}
	Since each $h_i$ is a homeomorphism, each $B\in \mathcal{B}$ is an open interval of $[a_k,b_k]$. Let $\{x_B,y_B\}$ be the endpoints of $B$. Since each $h_i$ is a homeomorphism, there exist $j_x,j_y\in \{0,...,k-1\}$ and $U_x,U_y\in \mathcal{ U}$ such that $x_B$ is an endpoint of $h_{k-1}^{j_x}(U_x)$ and  $y_B$ is an endpoint of $h_{k-1}^{j_y}(U_y)$. It follows that $B= h_{k-1}^{j_x}(U_x)\cap h_{k-1}^{j_y}(U_y)$. Since there are $k$ choices for both $j_x$ and $j_y$ and $|\mathcal{U}|$ choices for $U_x$ and $U_y$, there are $k^2|\mathcal{U}|^2$ choices for each $B$. Hence, $|\mathcal{B}|\leq k^2|\mathcal{U}|^2$. 
\end{proof}

\begin{corollary}\label{FEnt8} Let  $\{[a_i,b_i]\}_{i=1}^k$ be a collection of subintervals of $[0,1]$ and $\{h_i\}_{i=1}^{k-1}$ be a collection of homeomorphisms $h_i:[a_{i},b_{i}]\to [a_{i+1},b_{i+1}]$. Suppose $S\subset [0,1]^k$ is an $(k,\eps)$-separated set such that $h_i(x_i)=x_{i+1}$ for each $i$ and $( x_1,x_2,...,x_k )\in S$. Then $|S|\leq k^2(\frac{2}{\eps})^2$.
\end{corollary}
\begin{proof} First, there exists a cover $\mathcal{U}$ of $I$ of open intervals with no common endpoints such that $|\mathcal{U}|\leq 2/\eps$ and $\mbox{mesh}(\mathcal{U})<\eps$. Let ${\mathcal B}$ be a subcover of $\vee_{i=0}^{k-1}h_{k-1}^i(\mathcal{U})$ of minimal cardinality that covers $[a_k,b_k]$. Suppose $( x_1,x_2,...,x_k ), ( y_1,y_2,...,y_k )\in [0,1]^k$ are such that $h_i(x_i)=x_{i+1}$ and $h_i(y_i)=y_{i+1}$ for each $i$. If there exists a $B\in \mathcal{B}$ such that $x_k, y_k\in B$, then for each $i$ there exists $U_i\in {\mathcal U}$ such that $x_i, y_i\in U_i$. Hence, $\max_{1\leq i\leq k}\{{d}(x_i,y_i)\}<\eps$.  Thus, it follows from the pigeon-hole principle and Proposition~\ref{FEnt7} that $|S|\leq |\mathcal{B}|\leq k^2(\frac{2}{\eps})^2$.
\end{proof}

\begin{lemma}\label{FEnt9}  Suppose $n>m$ and $F_n\circ F_m^{-1}=F_m^{-1}\circ F_n=\Gamma$. Let $S$ be a  $(k,\eps)$-separated set for $\Gamma$. Then $|S|\leq k^2(\frac{2}{\eps})^2(k+1)n^k$,
\end{lemma}
\begin{proof}
	Let $S$ be a $(k,\eps)$-separated set. Let $M\in \mathcal{M}_k$ and define $S(M)=\{( x_1,x_2,...,x_k )\in S: (x_1,x_k)\in M\}$. Then there exist a collection of homeomorphisms $h_i\colon[p_i,q_i]\to [p_{i+1},q_{i+1}]$ such that 
	\begin{enumerate}
		\item $\pi^k_1(M)=[p_1,q_1]$
		\item $\pi^k_2(M)=[p_{k},q_{k}]$
		\item $x_i\in [p_i,q_i]$ for each $( x_1,x_2,...,x_k )\in S(M)$
		\item $h_i(x_i)=x_{i+1}$ for each $( x_1,x_2,...,x_k )\in S(M)$
	\end{enumerate} 
	Then it follows from Corollary~\ref{FEnt8}, that $|S(M)|\leq  k^2(\frac{2}{\eps})^2$. Hence $|S|\leq \sum_{M\in \mathcal{M}_k}|S(M)|\leq k^2(\frac{2}{\eps})^2(k+1)n^k$ by Lemma \ref{FEnt6}.
\end{proof}

\begin{theorem}\label{thm:upperbound} Suppose $n>m$ and $F_n\circ F_m^{-1}=F_m^{-1}\circ F_n=\Gamma$, then $\mbox{Ent}(\Gamma)\leq\log(n)$.
\end{theorem}
\begin{proof} Let $\eps>0$. Let $S_k$ be a $(k,\eps)$-separated set of largest cardinality. Then by Lemma \ref{FEnt9} we have
	\begin{eqnarray*}
	\limsup_{k\rightarrow \infty}\frac{\log(|S_k|)}{k} &\leq& \limsup_{k\rightarrow \infty}\frac{\log(k^2(\frac{2}{\eps})^2(k+1)n^k))}{k}\\
	&=&\log(n)+\limsup_{k\rightarrow \infty}\frac{\log(k^2(\frac{2}{\eps})^2(k+1))}{k}\\
	&=&\log(n).
	\end{eqnarray*}
	
	Then $\mbox{Ent}(\Gamma)=\lim_{\eps\rightarrow \infty}\limsup_{k\rightarrow \infty}\frac{\log(|S_k|)}{k}\leq \log(n)$.
\end{proof}

\begin{theorem}\label{thm:main}
	Let $F_n\circ F_m^{-1}=F_m^{-1}\circ F_n$. Then $$\Ent(F_n\circ F_m^{-1})=\max\{\log(n), \log(m)\}.$$
\end{theorem}
\begin{proof}
	Follows from Proposition~\ref{prop:lowerbound} and Theorem~\ref{thm:upperbound}.
\end{proof}

\begin{remark}
	Let $n\in\N$, $F_n$ be an open map with $n-1$ critical points which is {\em conjugate} to the symmetric tent map $T_n$, \ie there is a homeomorphism $h\colon I\to I$ such that $F_n=h^{-1}\circ T_n\circ h$. Then for every $m\in\N$ which is relatively prime to $n$ there exists $F_m$ as in Theorem~\ref{thm:main}. We simply take $F_m:=h^{-1}\circ T_m\circ h$, and use Propositions~3.3 and 3.4 from \cite{AM-stronglycomm}.
\end{remark}

\section{Strongly commuting interval maps}\label{sec:stronglycomm}

We assume that $f,g\colon I\to I$ are continuous, piecewise monotone, and satisfy $g\circ f^{-1}=f^{-1}\circ g$. We will show that $\Ent(g\circ f^{-1})=\max\{\Ent(f),\Ent(g)\}$. Recall the following theorem from \cite{AM-stronglycomm}:

\begin{theorem}[Theorem~5.21 in \cite{AM-stronglycomm}]\label{thm:strcom}
	Let $f,g\colon I\to I$ be piecewise monotone maps such that $f^{-1}\circ g=g\circ f^{-1}$. Then there are $0=p_0<p_1<\ldots<p_k=1$ such that $[p_i,p_{i+1}]$ is invariant under $f^2$ and $g^2$ for every $i\in\{0,1,\ldots,k-1\}$, and such that one of the following occurs:
	\begin{itemize}
			\item[(i)] $f^2|_{[p_i,p_{i+1}]}$ and $g^2|_{[p_i,p_{i+1}]}$ are both open and non-monotone,
			\item[(ii)] $f^2|_{[p_i,p_{i+1}]}$ is monotone, or
			\item[(iii)] $g^2|_{[p_i,p_{i+1}]}$ is monotone.
	\end{itemize}
\end{theorem}

Note that $\Ent(g^2\circ f^{-2})=\max\{\Ent(g^2\circ f^{-2}|_{[p_i,p_{i+1}]}): i\in\{0,1,\ldots,k-1\}\}$. 

\begin{theorem}[Theorem~1 in \cite{MisSl}]\label{thm:Mis}
	If $f\colon I\to I$ is a piecewise monotone map, then 
	$$\Ent(f)=\lim_{n\to\infty}\frac 1n\log(c_n),$$
	where $c_n$ denotes the number of critical points of $f^n$.
\end{theorem}

\begin{proposition}\label{prop:composition}
	If $f\colon I\to I$ is a piecewise monotone map, and $h\colon I\to I$ is a homeomorphism such that $f\circ h=h\circ f$, then $\Ent(f\circ h)=\Ent(h\circ f)=\Ent(f)$.
\end{proposition}
\begin{proof}
	Note that since $f$ and $h$ commute, we have $(f\circ h)^n=f^n\circ h^n$. Furthermore, if $C_n$ denotes the set of critical points of $f^n$, then $h^{-n}(C_n)$ is the set of critical points of $f^n\circ h^n$, and since $h^n$ is a homeomorphism $|C_n|=|h^{-n}(C_n)|$. Now Theorem~\ref{thm:Mis} gives
	$$\Ent(f\circ h)=\lim_{n\to\infty}\frac 1n\log(|h^{-n}(C_n)|)=\lim_{n\to\infty}\frac 1n\log(|C_n|)=\Ent(f).$$
\end{proof}

\begin{corollary}\label{cor:seconds}
	If $f,g\colon I\to I$ are piecewise monotone maps such that $g\circ f^{-1}=f^{-1}\circ g$, then $\Ent(g^2\circ f^{-2})=\max\{\Ent(g^2),\Ent(f^2)\}$.
\end{corollary}
\begin{proof}
	Let $0=p_0<p_1<\ldots<p_k=1$ be as in Theorem~\ref{thm:strcom}, and recall that $\Ent(g^2\circ f^{-2})=\max\{\Ent(g^2\circ f^{-2}|_{[p_i,p_{i+1}]}): i\in\{0,1,\ldots,k-1\}\}$. If $f^2|_{[p_i,p_{i+1}]}, g^2|_{[p_i,p_{i+1}]}$ are both open and non-monotone, then Theorem~\ref{thm:main} implies that 
	$$\Ent(g^2\circ f^{-2}|_{[p_i,p_{i+1}]}) = \max\{\Ent(g^2|_{[p_i,p_{i+1}]}),\Ent(f^2|_{[p_i,p_{i+1}]})\}.$$
	 If $g^2|_{[p_i,p_{i+1}]}$ is monotone, then $\Ent(g^2|_{[p_i,p_{i+1}]})=0$, and Proposition~\ref{prop:composition} implies that 
	$\Ent(f^2\circ g^{-2}|_{[p_i,p_{i+1}]})=\Ent(f^{2}|_{[p_i,p_{i+1}]})=\max\{\Ent(f^{2}|_{[p_i,p_{i+1}]}),0\}$. Since $\Ent(F)=\Ent(F^{-1})$, for every set-valued map $F\colon I\to I$, the claim follows. Similarly, if $f^2|_{[p_i,p_{i+1}]}$ is monotone, we get $\Ent(g^2\circ f^{-2}|_{[p_i,p_{i+1}]})=\max\{\Ent(g^2|_{[p_i,p_{i+1}]}),\Ent(f^2|_{[p_i,p_{i+1}]})\}$. Now, since we have $\Ent(f^2)=\max\{\Ent(f^2|_{[p_i,p_{i+1}]}): i\in\{0,1,\ldots,k-1\}\}$, and $\Ent(g^2)=\max\{\Ent(g^2|_{[p_i,p_{i+1}]}): i\in\{0,1,\ldots,k-1\}\}$, the proof is complete.
\end{proof}

\begin{corollary}\label{cor:main}
	If $f,g\colon I\to I$ are piecewise monotone maps such that $g\circ f^{-1}=f^{-1}\circ g$, then $\Ent(g\circ f^{-1})=\max\{\Ent(f),\Ent(g)\}$.
\end{corollary}
\begin{proof}
	Let $\Psi\colon\underleftarrow{\lim}(I,f)\to\underleftarrow{\lim}(I,f)$ be the diagonal map given by $\Psi((x_0,x_1,x_2,\ldots))=(g(x_1),g(x_2),\ldots)$, for every $(x_0,x_1,\ldots)\in\underleftarrow{\lim}(I,f)$. Since $g\circ f^{-1}=f^{-1}\circ g$, Proposition~\ref{prop:entdiag} implies that $\Ent(\Psi)=\Ent(g\circ f^{-1})$. 
	
	Let $\Phi\colon\underleftarrow{\lim}(I,f^2)\to\underleftarrow{\lim}(I,f^2)$ be the diagonal map given by $\Phi((y_0,y_1,y_2,\ldots))=(g^2(y_1),g^2(y_2),\ldots)$,  for every $(y_0,y_1,\ldots)\in\underleftarrow{\lim}(I,f^2)$. Then since also $g^2\circ f^{-2}=f^{-2}\circ g^2$, we have $\Ent(\Phi)=\Ent(g^2\circ f^{-2})$, and consequently by Corollary~\ref{cor:seconds}, $\Ent(\Phi)=\max\{\Ent(f^2),\Ent(g^2)\}$.
	
	Let $H\colon\underleftarrow{\lim}(I,f)\to\underleftarrow{\lim}(I,f^2)$ be a homeomorphism given by $H((x_0,x_1,x_2,\ldots))=(x_0,x_2,x_4,\ldots).$
	Note that $\Psi^2((x_0,x_1,x_2,\ldots))=(g^2(x_2),g^2(x_3),\ldots)$, so it follows that $H\circ\Psi^2((x_0,x_1,x_2,\ldots))=(g^2(x_2),g^2(x_4)\ldots)=\Phi(x_0,x_2,x_4,\ldots)=\Phi(H((x_0,x_1,x_2,\ldots))$, so $H\circ\Psi^2=\Phi\circ H$. In particular $\Phi$ and $\Psi^2$ are conjugated, and it follows that $\Ent(\Psi^2)=\Ent(\Phi)$.
	
	Now we have $2\Ent(g\circ f^{-1})=2\Ent(\Psi)=\Ent(\Psi^2)=\Ent(\Phi)=\Ent(g^2\circ f^{-2})=\max\{\Ent(g^2),\Ent(f^2)\}=2\max\{\Ent(g),\Ent(f)\}$, and thus it follows that $\Ent(g\circ f^{-1})=\max\{\Ent(g),\Ent(f)\}$.
\end{proof}

Recall that for set-valued functions $F$ in general $\Ent(F^n)\neq n\Ent(F)$. However, we have the following result:

\begin{corollary}
	If $f,g\colon I\to I$ are piecewise monotone maps such that $f\circ g^{-1}=g^{-1}\circ f$, then $\Ent((f\circ g^{-1})^n)=n\Ent(f\circ g^{-1})$.
\end{corollary}

\appendix

\section{An example}  

Let $a,b,c$ be such that $0\leq a \leq b\leq c$, $b=\log(n) \mbox{ or } \infty$ for some $n\in \N$ and $c=\log(m) \mbox{ or } \infty$ for some $m\in \N$.
 
We will construct maps $f_i,g_i:I\to I$ such that 
\begin{enumerate}
	\item $f_i\circ g_{i+1}=g_i\circ f_{i+1}$
	\item $\mbox{Ent}(\Psi)=a$
	\item $\liminf_{i\to\infty}\mbox{Ent}(g_i\circ f_{i}^{-1})=b$
	\item $\limsup_{i\to\infty}\mbox{Ent}(g_i\circ f_{i}^{-1})=c$
\end{enumerate}

Where ${X}=\underleftarrow{\lim}(I, f_i)$ and $\Psi\colon{X}\to {X}$ is a diagonal map defined by $\Psi((x_i )_{i=0}^{\infty})=( g_{i}(x_i) )_{i=1}^{\infty}$. That shows that if the assumption (ii) from Proposition~\ref{prop:entdiag} is not satisfied, then 
\begin{enumerate}
	\item[(i)] $\lim_{i\to\infty}\Ent(g_i\circ f_i^{-1})$ does not have to exist, and
	\item[(ii)] $\Ent(\Psi)$ can be strictly less than $\liminf_{i\to\infty}\Ent(g_i\circ f_i^{-1})$.
\end{enumerate}

Let $p$ be an odd natural number and $T_p$ be the symmetric tent map. Define $S_p$ to be a shifted 3-fold defined in the following way and dependent on $p$:

$$S_p(x)=\begin{cases}
\frac{2px}{p-1}, & 0\leq x\leq\frac{p-1}{2p}, \\
2-\frac{2px}{p-1}, &\frac{p-1}{2p}< x\leq\frac{p-1}{p},\\
px-p+1, &\frac{p-1}{p}<x\leq 1.
\end{cases}$$

For $p$ odd and $i\in \{1,...,p-2\}$ define $G_p:I\to I$

$$G_p(x)=\begin{cases}
\frac{(p-1)x}{2}, &0\leq x\leq \frac{2}{p} \\
x+\frac{p-1}{p}-\frac{i+1}{p}, &\frac{i+1}{p}< x\leq\frac{i+2}{p}, \text{$i$ odd}\\
1-x+\frac{i+2}{p}, &\frac{i+2}{p}< x\leq\frac{i+3}{p}, \text {$i$ even}.
\end{cases}$$

See Figure~\ref{fig:FpGp}.
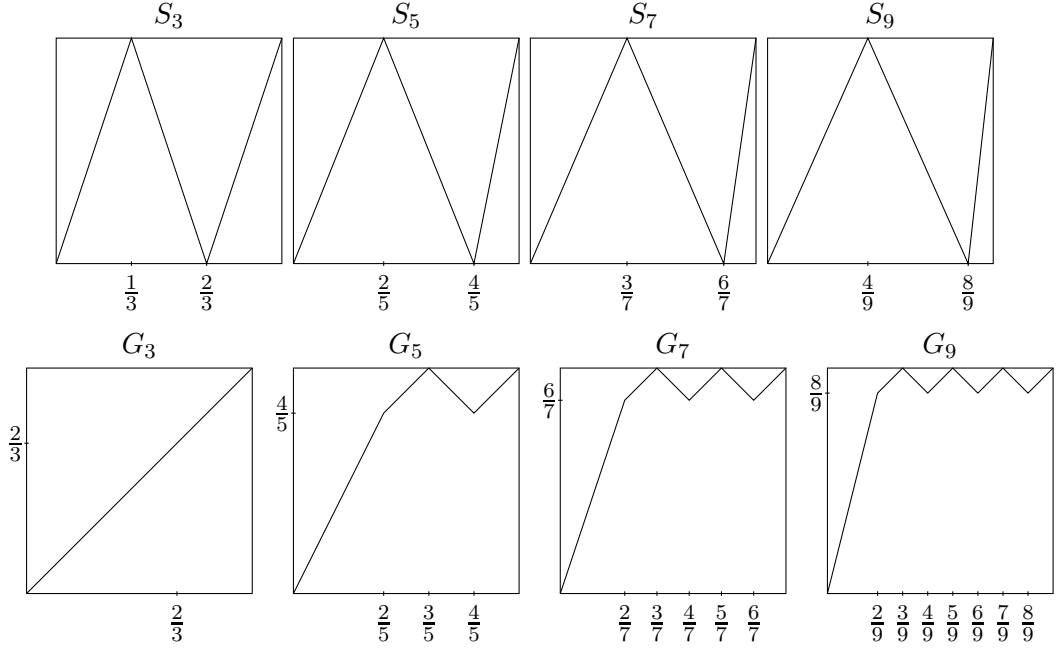
\begin{figure}[!ht]
	\centering
	\begin{tikzpicture}[scale=3]
	\draw (0,0)--(0,1)--(1,1)--(1,0)--(0,0);
	\draw (0,0)--(1/3,1)--(2/3,0)--(1,1);
	\node at (0.5,1.1) {\small $S_3$};
	\draw (1/3,-0.01)--(1/3,0.01);
	\node at (1/3,-0.13) {\small $\frac{1}{3}$};
	\draw (2/3,-0.01)--(2/3,0.01);
	\node at (2/3,-0.13) {\small $\frac{2}{3}$};
	\end{tikzpicture}
	\begin{tikzpicture}[scale=3]
	\draw (0,0)--(0,1)--(1,1)--(1,0)--(0,0);
	\draw (0,0)--(2/5,1)--(4/5,0)--(1,1);
	\node at (0.5,1.1) {\small $S_5$};
	\draw (2/5,-0.01)--(2/5,0.01);
	\node at (2/5,-0.13) {\small $\frac{2}{5}$};
	\draw (4/5,-0.01)--(4/5,0.01);
	\node at (4/5,-0.13) {\small $\frac{4}{5}$};
	\end{tikzpicture}
	\begin{tikzpicture}[scale=3]
	\draw (0,0)--(0,1)--(1,1)--(1,0)--(0,0);
	\draw (0,0)--(3/7,1)--(6/7,0)--(1,1);
	\node at (0.5,1.1) {\small $S_7$};
	\draw (3/7,-0.01)--(3/7,0.01);
	\node at (3/7,-0.13) {\small $\frac{3}{7}$};
	\draw (6/7,-0.01)--(6/7,0.01);
	\node at (6/7,-0.13) {\small $\frac{6}{7}$};
	\end{tikzpicture}
	\begin{tikzpicture}[scale=3]
	\draw (0,0)--(0,1)--(1,1)--(1,0)--(0,0);
	\draw (0,0)--(4/9,1)--(8/9,0)--(1,1);
	\node at (0.5,1.1) {\small $S_9$};
	\draw (4/9,-0.01)--(4/9,0.01);
	\node at (4/9,-0.13) {\small $\frac{4}{9}$};
	\draw (8/9,-0.01)--(8/9,0.01);
	\node at (8/9,-0.13) {\small $\frac{8}{9}$};
	\end{tikzpicture}
	\begin{tikzpicture}[scale=3]
	\draw (0,0)--(0,1)--(1,1)--(1,0)--(0,0);
	\draw (0,0)--(1,1);
	\node at (0.5,1.1) {\small $G_3$};
	\draw (2/3,-0.01)--(2/3,0.01);
	\node at (2/3,-0.13) {\small $\frac{2}{3}$};
	\draw (-0.01,2/3)--(0.01,2/3);
	\node at (-0.05,2/3) {\small $\frac{2}{3}$};
	\end{tikzpicture}
	\begin{tikzpicture}[scale=3]
	\draw (0,0)--(0,1)--(1,1)--(1,0)--(0,0);
	\draw (0,0)--(2/5,4/5)--(3/5,1)--(4/5,4/5)--(1,1);
	\node at (0.5,1.1) {\small $G_5$};
	\draw (2/5,-0.01)--(2/5,0.01);
	\node at (2/5,-0.13) {\small $\frac{2}{5}$};
	\draw (3/5,-0.01)--(3/5,0.01);
	\node at (3/5,-0.13) {\small $\frac{3}{5}$};
	\draw (4/5,-0.01)--(4/5,0.01);
	\node at (4/5,-0.13) {\small $\frac{4}{5}$};
	\draw (-0.01,4/5)--(0.01,4/5);
	\node at (-0.05,4/5) {\small $\frac{4}{5}$};
	\end{tikzpicture}
	\begin{tikzpicture}[scale=3]
	\draw (0,0)--(0,1)--(1,1)--(1,0)--(0,0);
	\draw (0,0)--(2/7,6/7)--(3/7,1)--(4/7,6/7)--(5/7,1)--(6/7,6/7)--(1,1);
	\node at (0.5,1.1) {\small $G_7$};
	\draw (2/7,-0.01)--(2/7,0.01);
	\node at (2/7,-0.13) {\small $\frac{2}{7}$};
	\draw (3/7,-0.01)--(3/7,0.01);
	\node at (3/7,-0.13) {\small $\frac{3}{7}$};
	\draw (4/7,-0.01)--(4/7,0.01);
	\node at (4/7,-0.13) {\small $\frac{4}{7}$};
	\draw (5/7,-0.01)--(5/7,0.01);
	\node at (5/7,-0.13) {\small $\frac{5}{7}$};
	\draw (6/7,-0.01)--(6/7,0.01);
	\node at (6/7,-0.13) {\small $\frac{6}{7}$};
	\draw (-0.01,6/7)--(0.01,6/7);
	\node at (-0.05,6/7) {\small $\frac{6}{7}$};
	\end{tikzpicture}
	\begin{tikzpicture}[scale=3]
	\draw (0,0)--(0,1)--(1,1)--(1,0)--(0,0);
	\draw (0,0)--(2/9,8/9)--(3/9,1)--(4/9,8/9)--(5/9,1)--(6/9,8/9)--(7/9,1)--(8/9,8/9)--(1,1);
	\node at (0.5,1.1) {\small $G_9$};
	\draw (2/9,-0.01)--(2/9,0.01);
	\node at (2/9,-0.13) {\small $\frac{2}{9}$};
	\draw (3/9,-0.01)--(3/9,0.01);
	\node at (3/9,-0.13) {\small $\frac{3}{9}$};
	\draw (4/9,-0.01)--(4/9,0.01);
	\node at (4/9,-0.13) {\small $\frac{4}{9}$};
	\draw (5/9,-0.01)--(5/9,0.01);
	\node at (5/9,-0.13) {\small $\frac{5}{9}$};
	\draw (6/9,-0.01)--(6/9,0.01);
	\node at (6/9,-0.13) {\small $\frac{6}{9}$};
	\draw (7/9,-0.01)--(7/9,0.01);
	\node at (7/9,-0.13) {\small $\frac{7}{9}$};
	\draw (8/9,-0.01)--(8/9,0.01);
	\node at (8/9,-0.13) {\small $\frac{8}{9}$};
	\draw (-0.01,8/9)--(0.01,8/9);
	\node at (-0.05,8/9) {\small $\frac{8}{9}$};
	\end{tikzpicture}
	\caption{Graphs of maps $S_p$ and $G_p$ for $p=3,5,7,9$.}
	\label{fig:FpGp}
\end{figure}

Notice that $S_p\circ G_p=T_p=id\circ T_p$. Also, $T_p\circ S_p^{-1}=T_{\frac{p-1}{2}}\cup D$, where $D=\{(x,x): x\in I\}$ is the {\it diagonal}, see Figure~\ref{fig:TpGpinv}. So $S_p\circ T_p^{-1}=T_{\frac{p-1}{2}}^{-1}\cup D$.
Hence, 
\begin{equation}\label{eq:ent}
\mbox{Ent}(S_p\circ T_p^{-1})=\log(\frac{p-1}{2}).
\end{equation}
Furthermore, $G_p(x)\geq x$ for all $\in I$. Hence $\mbox{Ent}(G_p)=0$. Thus, $\mbox{Ent}(G_p^{-1})=0$.  

\begin{figure}[!ht]
	\centering
	\begin{tikzpicture}[scale=3]
	\draw (0,0)--(0,1)--(1,1)--(1,0)--(0,0);
	\node at (0.5,1.1) {\small $T_5\circ S_5^{-1}$};
	\draw (1,1)--(0,0)--(1/2,1)--(1,0);
	\draw (1/2,-0.01)--(1/2,0.01);
	\node at (1/2,-0.13) {\small $\frac{1}{2}$};
	\end{tikzpicture}
	\begin{tikzpicture}[scale=3]
	\draw (0,0)--(0,1)--(1,1)--(1,0)--(0,0);
	\node at (0.5,1.1) {\small $T_7\circ S_7^{-1}$};
	\draw (0,0)--(1/3,1)--(2/3,0)--(1,1)--(0,0);
	\draw (1/3,-0.01)--(1/3,0.01);
	\node at (1/3,-0.13) {\small $\frac{1}{3}$};
	\draw (2/3,-0.01)--(2/3,0.01);
	\node at (2/3,-0.13) {\small $\frac{2}{3}$};
	\end{tikzpicture}
	\begin{tikzpicture}[scale=3]
	\draw (0,0)--(0,1)--(1,1)--(1,0)--(0,0);
	\node at (0.5,1.1) {\small $T_9\circ S_9^{-1}$};
	\draw (1,1)--(0,0)--(1/4,1)--(2/4,0)--(3/4,1)--(1,0);
	\draw (1/4,-0.01)--(1/4,0.01);
	\node at (1/4,-0.13) {\small $\frac{1}{4}$};
	\draw (2/4,-0.01)--(2/4,0.01);
	\node at (2/4,-0.13) {\small $\frac{1}{2}$};
	\draw (3/4,-0.01)--(3/4,0.01);
	\node at (3/4,-0.13) {\small $\frac{3}{4}$};
	\end{tikzpicture}
	\caption{Graphs of $T_p\circ S_p^{-1}$ for $p=5,7,9$.}
	\label{fig:TpGpinv}
\end{figure}
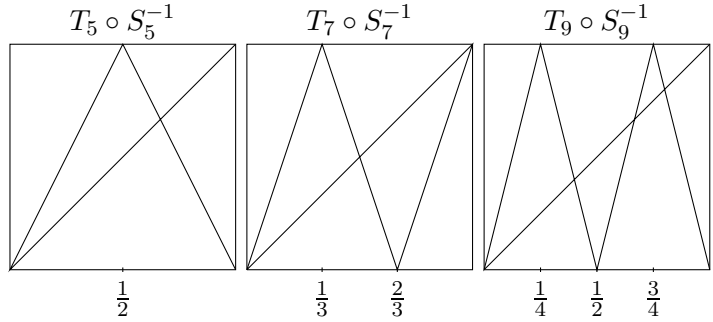

Let $R:I\to I$ be defined by 
$$R(x)=\begin{cases}
\frac{3}{2}x, & x\in [0,\frac{1}{3}]  \\
\frac{1}{2} & x\in (\frac{1}{3},\frac{2}{3})\\
\frac{3}{2}x-\frac{1}{2}, & x\in [\frac{2}{3},1].
\end{cases}$$

Let $L_{a,b}:I\to[a,b]$ be a homeomorphism defined by $L_{a,b}(x)=(b-a)x+a$. 

Next if $w:I\to I$, then define $\widetilde{w}:I\to I$ by

$$\widetilde{w}(x)=\begin{cases}
x, & x\in [0,\frac{1}{3})\cup(\frac{2}{3},1]\\
L_{\frac{1}{3},\frac{2}{3}}\circ w\circ L^{-1}_{\frac{1}{3},\frac{2}{3}}(x), & x\in [\frac{2}{3},1].
\end{cases}$$

See Figure~\ref{fig:Rwtilde}.
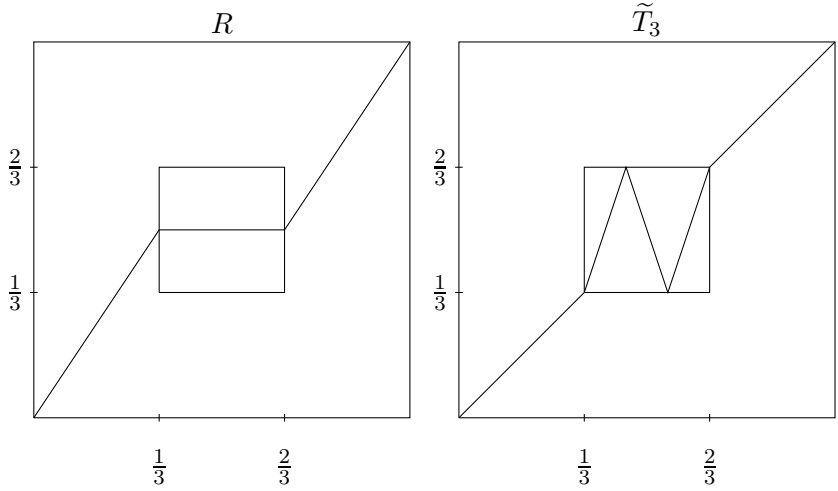
\begin{figure}[!ht]
	\centering
	\begin{tikzpicture}[scale=5]
	\draw (0,0)--(0,1)--(1,1)--(1,0)--(0,0);
	\draw[thin] (1/3,1/3)--(1/3,2/3)--(2/3,2/3)--(2/3,1/3)--(1/3,1/3);
	\node at (0.5,1.05) {\small $R$};
	\draw (0,0)--(1/3,1/2)--(2/3,1/2)--(1,1);
	\draw (1/3,-0.01)--(1/3,0.01);
	\node at (1/3,-0.13) {\small $\frac{1}{3}$};
	\draw (2/3,-0.01)--(2/3,0.01);
	\node at (2/3,-0.13) {\small $\frac{2}{3}$};
	\draw (-0.01,1/3)--(0.01,1/3);
	\node at (-0.05,1/3) {\small $\frac{1}{3}$};
	\draw (-0.01,2/3)--(0.01,2/3);
	\node at (-0.05,2/3) {\small $\frac{2}{3}$};
	\end{tikzpicture}
	\begin{tikzpicture}[scale=5]
	\draw (0,0)--(0,1)--(1,1)--(1,0)--(0,0);
	\draw[thin] (1/3,1/3)--(1/3,2/3)--(2/3,2/3)--(2/3,1/3)--(1/3,1/3);
	\node at (0.5,1.06) {\small $\widetilde T_3$};
	\draw (0,0)--(1/3,1/3)--(1/3+1/9,2/3)--(1/3+2/9,1/3)--(2/3,2/3)--(1,1);
	\draw (1/3,-0.01)--(1/3,0.01);
	\node at (1/3,-0.13) {\small $\frac{1}{3}$};
	\draw (2/3,-0.01)--(2/3,0.01);
	\node at (2/3,-0.13) {\small $\frac{2}{3}$};
	\draw (-0.01,1/3)--(0.01,1/3);
	\node at (-0.05,1/3) {\small $\frac{1}{3}$};
	\draw (-0.01,2/3)--(0.01,2/3);
	\node at (-0.05,2/3) {\small $\frac{2}{3}$};
	\end{tikzpicture}
	\caption{Maps $R$ and $\tilde T_3$.}
	\label{fig:Rwtilde}
\end{figure}

Also, for $i\in\N$ define ${\mathcal L}_i(f):[\frac{2^{i-1}-1}{2^{i-1}},\frac{2^{i}-1}{2^i}]\to [\frac{2^{i-1}-1}{2^{i-1}},\frac{2^{i}-1}{2^i}] $ by  \[{\mathcal L}_i(f)(x)=L_{\frac{2^{i-1}-1}{2^{i-1}},\frac{2^{i}-1}{2^i}}\circ f\circ L_{\frac{2^{i-1}-1}{2^{i-1}},\frac{2^{i}-1}{2^i}}^{-1}(x).\]

Note that if $f\circ g=g\circ f$ then  $\tilde f\circ \tilde g=\tilde g\circ \tilde f$ and ${\mathcal L}_i(f)\circ {\mathcal L}_i(g)={\mathcal L}_i(g)\circ {\mathcal L}_i(f)$. Also, $\mbox{Ent}(\tilde f)=\mbox{Ent}(f)=\mbox{Ent}({\mathcal L}_i(f))$ and $\mbox{Ent}(\tilde g\circ \tilde f^{-1})=\mbox{Ent}(g\circ f^{-1})=\mbox{Ent}({\mathcal L}_i(g)\circ ({\mathcal L}_i(f))^{-1})$. Furthermore if $f:I \to I$ is any function then,  $R\circ \tilde{f}=R$.

\begin{theorem}[See maps $g_{\eps}$ in the proof of Theorem~15 in \cite{Mou-entropy}]\label{thm:anyent} For every $a\in [0,\infty]$ there exists a map $W_a:I\to I$ such that $W_a(0)=0$, $W_a(1)=1$, and $\Ent(W_a)=a$. 
\end{theorem}

\begin{remark}
	Maps $g_{\eps}\colon I\to I$ from \cite{Mou-entropy} have infinitely many pieces of monotonicity. If $a<\infty$, it is also possible to construct $W_a$ as in Theorem~\ref{thm:anyent} which is additionally piecewise monotone. We first recall that if $f\colon I\to I$ is piecewise linear, and such that the absolute value of the slope of every piece equals $s$ for some $s\geq 1$, then $\Ent(f)=\log(s)$. This follows from \cite{MisSl}, once we note that $Var(f^n)=\sup\{\sum_{i=1}^k|f^n(x_{i-1})-f^n(x_{i})|, 0=x_0<x_1<\ldots<x_n=1\}=s^n$, for every $n\in\N$. Given $0\leq a<\infty$, we construct a piecewise linear $f_a\colon [0,1/2]\to[0,1/2]$, such that $f_a(0)=0$, and the absolute value of the slopes of pieces of $f_a$ equal $e^a$. We then define
	$$W_a(x)=\begin{cases}
	f_a(x), & x\in[0,1/2],\\
	2(x-1)(1-f_a(1/2))+1, & x\in[1/2,1],
	\end{cases}$$
	and use Theorem~12 in \cite{Mou-entropy} to conclude that $\Ent(W_a)=\Ent(f_a)=a$. Since every piecewise monotone map has finite entropy, note that $W_{\infty}$ has to have infinitely many monotonicity pieces.
\end{remark}

Let $\{n_i\}_{i=1}^{\infty}$ be a sequence of natural numbers such that $\liminf_{i\rightarrow\infty}\log(n_i)=b$ and $\limsup_{i\rightarrow\infty}\log(n_i)=c$ and choose $W_a$ from the above theorem.

Now define (see Figure~\ref{fig:gi} and Figure~\ref{fig:fi}):

\noindent\begin{minipage}{.5\linewidth}
	$$g_1(x)=\begin{cases}
	{\mathcal L}_1(W_a)(x), & x\in [0,\frac{1}{2}]\\
	{\mathcal L}_2(\widetilde{T}_{2n_1+1})(x), & x\in [\frac{1}{2},\frac{3}{4}]\\
	{\mathcal L}_3(R)(x), & x\in [\frac{3}{4},\frac{7}{8}]\\
	x, & \frac{7}{8}<x.
	\end{cases}$$
\end{minipage}
\begin{minipage}{.5\linewidth}
	$$f_1(x)=\begin{cases}
	x, & x\in [0,\frac{1}{2}]\\
	{\mathcal L}_2(\widetilde{S}_{2n_1+1})(x), & x\in [\frac{1}{2},\frac{3}{4}]\\
	{\mathcal L}_3(R)(x), & x\in [\frac{3}{4},\frac{7}{8}]\\
	x, & \frac{7}{8}<x.
	\end{cases},$$
\end{minipage}

and for $k>1$ define

	$$g_k(x)=\begin{cases}
	{\mathcal L}_1(W_a)(x), & x\in [0,\frac{1}{2}]\\
	{\mathcal L}_{i+1}(\widetilde{G}_{2n_i+1})(x), & x\in [\frac{2^{i}-1}{2^{i}},\frac{2^{i+1}-1}{2^{i+1}}] \mbox{ and } i\in\{1,...,k-1\}\\
	{\mathcal L}_{k+1}(\widetilde{T}_{2n_k+1})(x), & x\in [\frac{2^{k}-1}{2^{k}},\frac{2^{k+1}-1}{2^{k+1}}]\\
	{\mathcal L}_{k+2}(R)(x), & x\in [\frac{2^{k+1}-1}{2^{k+1}},\frac{2^{k+2}-1}{2^{k+2}}]\\
	x, & \frac{2^{k+2}-1}{2^{k+2}}<x.
	\end{cases}$$
	$$f_k(x)=\begin{cases}
	x, & x\in [0,\frac{2^{k}-1}{2^{k}}]\\
	{\mathcal L}_{k+1}(\widetilde{S}_{2n_k+1})(x), & x\in [\frac{2^{k}-1}{2^{k}},\frac{2^{k+1}-1}{2^{k+1}}]\\
	{\mathcal L}_{k+2}(R)(x), & x\in [\frac{2^{k+1}-1}{2^{k+1}},\frac{2^{k+2}-1}{2^{k+2}}]\\
	x, & \frac{2^{k+2}-1}{2^{k+2}}<x.
	\end{cases}$$

\begin{figure}[!ht]
	\centering
	\begin{tikzpicture}[scale=7.7]
	\draw (0,0)--(0,1)--(1,1)--(1,0)--(0,0);
	\draw[thin] (0,1/2)--(1,1/2);
	\draw[thin] (1/2,0)--(1/2,1);
	\draw[thin] (1/2,3/4)--(1,3/4);
	\draw[thin] (3/4,1/2)--(3/4,1);
	\draw[thin] (3/4,7/8)--(1,7/8);
	\draw[thin] (7/8,3/4)--(7/8,1);
	\node at (0.5,1.05) {\small $g_1$};
	\node at (1/4,1/4) {\small $W_a$};
	\node at (5/8,5/8) {\small $\widetilde T_{2n_1+1}$};
	\node at (13/16,13/16) {\small $R$};
	\draw (7/8,7/8)--(1,1);
	\end{tikzpicture}
	\begin{tikzpicture}[scale=7.7]
	\draw (0,0)--(0,1)--(1,1)--(1,0)--(0,0);
	\draw[thin] (0,1/2)--(1,1/2);
	\draw[thin] (1/2,0)--(1/2,1);
	\draw[thin] (1/2,3/4)--(1,3/4);
	\draw[thin] (3/4,1/2)--(3/4,1);
	\draw[thin] (3/4,7/8)--(1,7/8);
	\draw[thin] (7/8,3/4)--(7/8,1);
	\draw[thin] (7/8,15/16)--(1,15/16);
	\draw[thin] (15/16,7/8)--(15/16,1);
	\node at (0.5,1.05) {\small $g_2$};
	\node at (1/4,1/4) {\small $W_a$};
	\node at (5/8,5/8) {\small $\widetilde G_{2n_1+1}$};
	\node at (13/16,13/16) {\tiny $\widetilde T_{2n_2+1}$};
	\node at (29/32,29/32) {\small $R$};
	\draw (15/16,15/16)--(1,1);
	\end{tikzpicture}
	\caption{Depiction of graphs of $g_1$ and $g_2$.}
	\label{fig:gi}
\end{figure}
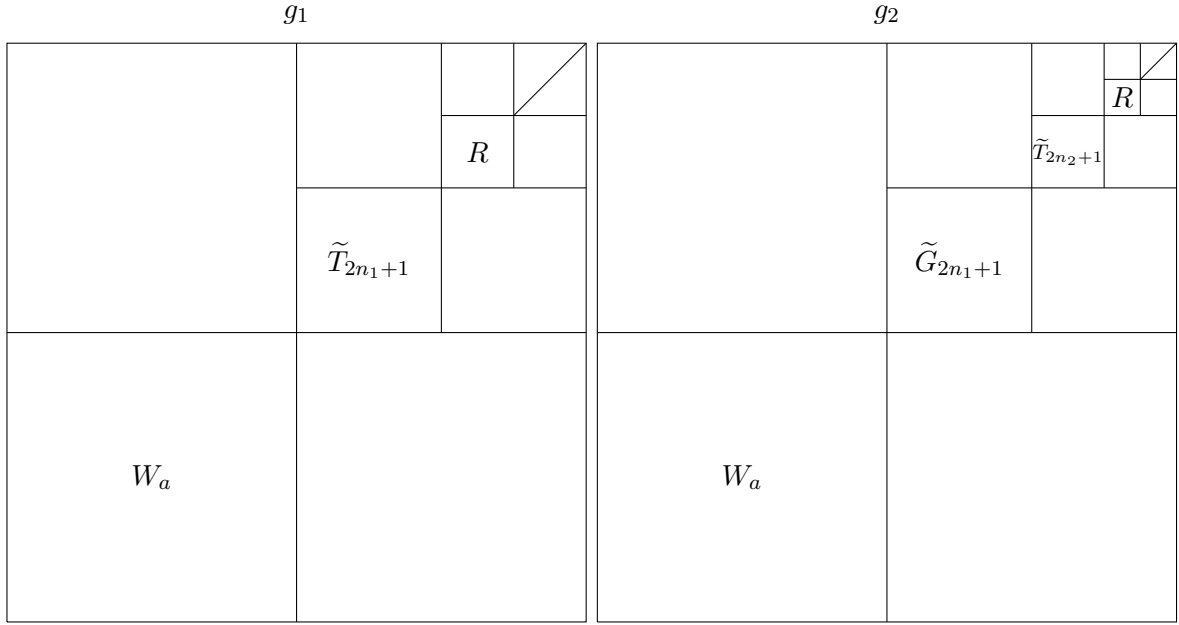

\begin{figure}[!ht]
	\centering
	\begin{tikzpicture}[scale=7.7]
	\draw (0,0)--(0,1)--(1,1)--(1,0)--(0,0);
	\draw[thin] (0,1/2)--(1,1/2);
	\draw[thin] (1/2,0)--(1/2,1);
	\draw[thin] (1/2,3/4)--(1,3/4);
	\draw[thin] (3/4,1/2)--(3/4,1);
	\draw[thin] (3/4,7/8)--(1,7/8);
	\draw[thin] (7/8,3/4)--(7/8,1);
	\node at (0.5,1.05) {\small $f_1$};
	\draw (0,0)--(1/2,1/2);
	\node at (5/8,5/8) {\small $\widetilde S_{2n_1+1}$};
	\node at (13/16,13/16) {\small $R$};
	\draw (7/8,7/8)--(1,1);
	\end{tikzpicture}
	\begin{tikzpicture}[scale=7.7]
	\draw (0,0)--(0,1)--(1,1)--(1,0)--(0,0);
	\draw[thin] (0,1/2)--(1,1/2);
	\draw[thin] (1/2,0)--(1/2,1);
	\draw[thin] (1/2,3/4)--(1,3/4);
	\draw[thin] (3/4,1/2)--(3/4,1);
	\draw[thin] (3/4,7/8)--(1,7/8);
	\draw[thin] (7/8,3/4)--(7/8,1);
	\draw[thin] (7/8,15/16)--(1,15/16);
	\draw[thin] (15/16,7/8)--(15/16,1);
	\node at (0.5,1.05) {\small $f_2$};
	\draw (0,0)--(3/4,3/4);
	\node at (13/16,13/16) {\tiny $\widetilde S_{2n_2+1}$};
	\node at (29/32,29/32) {\small $R$};
	\draw (15/16,15/16)--(1,1);
	\end{tikzpicture}
	\caption{Depiction of graphs of $f_1$ and $f_2$.}
	\label{fig:fi}
\end{figure}
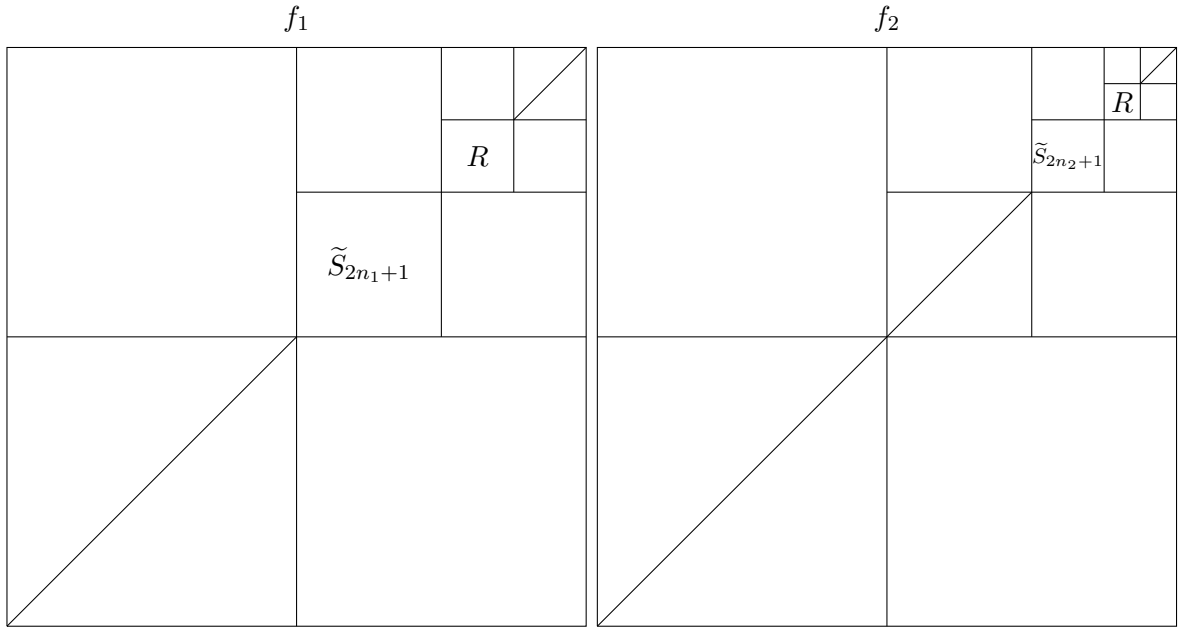

Then $f_k\circ g_{k+1}=g_k\circ f_{k+1}$ for every $k\in\N$ so we can define the diagonal map $\Psi\colon X\to  X$, $\Psi((x_i)_{i=0}^{\infty})=(g_i(x_i))_{i=1}^{\infty}$. Theorem~\ref{thm:liminf} implies that $\Ent(\Psi)\leq\liminf_{k\to\infty}\Ent(\psi_k)$, where $\psi_k=g_k\circ f_k^{-1}$ for every $k\in\N$.

Note the following:
\begin{enumerate}
	\item $g_k([\frac{2^{i}-1}{2^{i}},\frac{2^{i+1}-1}{2^{i+1}}])=[\frac{2^{i}-1}{2^{i}},\frac{2^{i+1}-1}{2^{i+1}}]=f_k([\frac{2^{i}-1}{2^{i}},\frac{2^{i+1}-1}{2^{i+1}}])$ and $g_k(1)=1=f_k(1)$  for all $k$ and $i$.

	\item By (1) it make sense to define ${X}_n=\underleftarrow{\lim}([\frac{2^{n-1}-1}{2^{n-1}},\frac{2^{n}-1}{2^{n}}], f_i)$ for $n\in {\mathbb N}$.
	
	\item Each ${X}_n$ is invariant under $\Psi$.
	
	\item ${X}=\bigcup_{n=1}^{\infty}{X}_n$ so $\mbox{Ent}(\Psi)=\sup_n \mbox{Ent}(\Psi|_{{X}_n})$.
	
	\item $\mbox{Ent}(\psi_i|_{[0,\frac{1}{2}]})=\mbox{Ent}(W_a)=a$ for all $i$, so $\mbox{Ent}(\Psi|_{{X}_1})=a$.
	
	\item For $n\geq 2$, $\mbox{Ent}(\psi_i|_{[\frac{2^{n-1}-1}{2^{n-1}},\frac{2^{n}-1}{2^{n}}]})=0$ for all $i>n$. So  $$\mbox{Ent}(\Psi|_{{X}_n})\leq \liminf_{i\rightarrow \infty}\mbox{Ent}(\psi_i|_{[\frac{2^{n-1}-1}{2^{n-1}},\frac{2^{n}-1}{2^{n}}]})=0.$$
	It follows that $\mbox{Ent}(\Psi)=\mbox{Ent}(\Psi|_{{X}_1})=a$.
	
	\item By (\ref{eq:ent}), $\mbox{Ent}(\psi_k)=\max\{a, \log{n_k}\}$. So $\liminf_{k\rightarrow \infty}\mbox{Ent}(\psi_k)=\liminf_{k\rightarrow \infty}\log{n_k}=b$ and $\limsup_{k\rightarrow \infty}\mbox{Ent}(\psi_k)=\limsup_{k\rightarrow \infty}\log{n_k}=c$.
\end{enumerate}


\begin{thebibliography}{A}
	
	\bibitem{AdKoMcA} R.\ L.\ Adler, A.\ G.\ Konheim, M.\ H.\ McAndrew, {\em Topological entropy}, Trans.\ Amer.\ Math.\ Soc.\, {\bf 114} (1965), 309--319.
	
	\bibitem{AlKel} L.\ Alvin, J.\ Kelly, {\em Topological entropy of Markov set-valued functions}, Ergodic Theory and Dynamical Systems (2019), 17pp.
	
	\bibitem{AM-stronglycomm} A.\ Anu\v si\'c, C.\ Mouron, {\em Strongly commuting interval maps}, preprint 2020.
	
	\bibitem{Bowen-ent} R.\ Bowen, {\em Entropy for group endomorphisms and homogeneous spaces}, Trans.\ Amer.\ Math.\ Soc.\, {\bf 153} (1971), 401--414.
	
	\bibitem{Bowen} R.\ Bowen, {\em Topological entropy and axiom A}, in: Proc.\ Sympos.\ Pure Math.\ XIV (1970) 23--41.
	
	\bibitem{BrSt} H.\ Bruin, S.\ \v Stimac, {\em Entropy of homeomorphisms on unimodal inverse limit spaces}, Nonlinearity {\bf 26} (2013), 991--1000.
	
	\bibitem{COAR} D.\ Carrasco-Olivera, R.\ Metzger Alvan, A.\ Morales Rojas, {\em Topological entropy for set-valued maps}, DCDS-B {\bf 20} (2015), 3461--3474. 
	
	\bibitem{Dinaburg}  E.\ I.\ Dinaburg, {\em A correlation between topological entropy and metric entropy (Russian)}, Dokl.\ Akad.\ Nauk SSSR, {\bf 190} (1970), 19--22.
	
	\bibitem{ErcegKen} G.\ Erceg, J.\ Kennedy, {\em Topological  entropy  on  closed sets in $[0,1]^2$}, Topology Appl.\ {\bf 246} (2018), 106--136.
	
	\bibitem{Freudenthal} H.\ Freudenthal, {\em Entwicklungen von R\"aumen und ihren Gruppen}, Compositio Math.\, vol. 4 (1937), 145--234.
	
	\bibitem{HoHerGut} L.\ Hoehn, R.\ Hern\'andez-Guti\'errez, {\em A fixed-point-free map of a tree-like continuum induced by bounded valence maps on trees}, Colloq.\ Math.\ {\bf 151} (2018), no. 2, 305--316. 
	
	\bibitem{IngMah} W.\ T.\ Ingram, W.\ S.\ Mahavier, {\em Inverse limits of upper semi-continuous set valued functions}, Houston J.\ Math.\ {\bf 32} (1) (2006), 119--130.
	
	\bibitem{KelTen} J.\ P.\ Kelly, T.\ Tennant, {\em Topological entropy of set-valued functions}, Houston J.\ Math.\ {\bf 43} (1) (2017), 263--282.
	
	\bibitem{JudyVan} J.\ Kennedy, V.\ Nall, {\em Dynamical properties of shift maps on inverse limits with a set valued function}, Ergod.\ Th.\ \& Dynam.\ Sys.\ {\bf 38} (4) (2018), 1499--1524.
	
	\bibitem{Mah} W.\ S.\ Mahavier, {\em Inverse limits with subsets of $[0,1]\times [0,1]$}, Topology Appl.\ {\bf 141} (1–3) (2004), 225--231.
	
	\bibitem{Miodus} J.\ Mioduszewski, {\em Mappings of inverse limits}, Colloquium Mathematicum {\bf 10} (1963), 39--44.
	
	\bibitem{MisSl} M.\ Misiurewicz, W.\ Szlenk, {\em Entropy of piecewise monotone mappings}, Studia Math.\ {\bf 67} (1980), 45--63. 
	
	\bibitem{Mou2} C.\ Mouron, {\em Exact maps of the pseudo-arc},	The 52th Spring Topology and Dynamical Systems	Conference: Auburn, AL, March 14-17, 2018.
	
	\bibitem{Mou-psarc} C.\ Mouron, {\em Entropy of shift maps of the pseudo-arc}, Topology Appl.\ {\bf 159} (2012), 34--39.
	
	\bibitem{Mou-entropy} C.\ Mouron, {\em A chainable continuum that admits a homeomorphism with entropy of arbitrary value}, Houston J.\ Math.\ {\bf 35} (4) (2009), 1079--1090.
	
	\bibitem{OvRog} L.\ G.\  Oversteegen, J.\ T.\ Rogers, Jr., {\em An  inverse  limit description of an atriodic tree-like continuum and an induced map without a fixed point}, Houston J.\ Math.\ {\bf 6} (1980), no. 4, 549--564.
	
	\bibitem{OvRog2} L.\ G.\  Oversteegen, J.\ T.\ Rogers, Jr., {\em Fixed-point-free maps on tree-like continua}, Topology Appl.\ {\bf 13} (1982), no. 1, 85--95.
	
	\bibitem{Ye} X.\ Ye, {\em Topological entropy of the induced maps of the inverse limits with bonding maps}, Topology Appl.\ {\bf 67} (1995), 113--118.
	

	

	
	
\end{thebibliography}
\end{document}